\documentclass[12pt]{article}
\usepackage[left=1in,top=1in,right=1in,bottom=1in,nohead]{geometry}

\usepackage{url,subfig,latexsym,amsmath,amssymb, amsfonts,enumerate,
  epsfig, graphics,times, amsthm,mathtools, bbm}

\usepackage{tikz, pgfplots}
\usetikzlibrary{automata,positioning,arrows, intersections, calc}
\pgfplotsset{compat=1.8}

\usepackage{commath}
\usepackage{wrapfig}
\usepackage{xcolor}
\usepackage{nameref}
\usepackage{listings}
\usepackage[colorlinks]{hyperref}
\usepackage{makeidx}
\usepackage{tabularx}
\usepackage{array}
    \newcolumntype{P}[1]{>{\centering\arraybackslash}p{#1}}
    \newcolumntype{M}[1]{>{\centering\arraybackslash}m{#1}}

     \providecommand{\classification}[1]{\textbf{AMS subject classifications:} #1}
\usetikzlibrary {graphs}
\usepackage{chemfig}
\newtheorem{thm}{Theorem}[section]
\newtheorem{cor}[thm]{Corollary}
\newtheorem{lem}[thm]{Lemma}
\newtheorem{prop}[thm]{Proposition}

\newtheorem*{thm*}{Theorem}

\theoremstyle{definition}
\newtheorem{defn}{Definition}[section]
\newtheorem{rem}{Remark}[section]

\newtheorem{example}{Example}[section]

\def\S{\mathbb{S}}

\newcommand{\dlim}{\displaystyle \lim\limits}
\newcommand{\dsum}{\displaystyle \sum\limits}

\def\T1{T^1_{\{x_n\}}}
\def\Sp{\mathcal S}
\def\C{\mathcal C}
\def\Re{\mathcal R}
\def\R{\mathbb R}

\def\Z{\mathbb Z}

\def\T1{T^1_{\{x_n\}}}
\def\Td1{T^{D,1}_{\{x_n\}}}
\def\Ts1{T^{S,1}_{\{x_n\}}}

\def\TV{\text{TV}}

\def\Tid1{T^{D,1}_{\{x_n+\zeta_i\}}}
\def\Tis1{T^{S,1}_{\{x_n+\zeta_i\}}}
\def\Tjs1{T^{S,1}_{\{x_n+\zeta_j\}}}
\def\Tjd1{T^{D,1}_{\{x_n+\zeta_j\}}}

\newcount\colveccount
\newcommand*\colvec[1]{
        \global\colveccount#1
        \begin{pmatrix}
        \colvecnext
}
\def\colvecnext#1{
        #1
        \global\advance\colveccount-1
        \ifnum\colveccount>0
                \\
                \expandafter\colvecnext
        \else
                \end{pmatrix}
        \fi
}

\makeatletter
\newcommand{\labitem}[2]{%
\def\@itemlabel{\textbf{#1}}
\item
\def\@currentlabel{#1}\label{#2}}
\makeatother

\title{
 A path method for non-exponential ergodicity of Markov chains and its application for chemical reaction systems}

\author{
Minjoon Kim \and
Jinsu Kim  }

\begin{document}

\tikzset{every node/.style={font=\tiny,sloped,above}}
\tikzset{every state/.style={rectangle, minimum size=0pt, draw=none, font=\normalsize}}
 \tikzset{bend angle=15}
 
\maketitle

\begin{abstract}

In this paper, we present criteria for non-exponential ergodicity of continuous-time Markov chains on a countable state space. These criteria can be verified by examining the ratio of transition rates over certain paths.
We applied this path method to explore the non-exponential convergence of microscopic biochemical interacting systems. Using reaction network descriptions, we identified special architectures of biochemical systems for non-exponential ergodicity. In essence, we found that reactions forming a cycle in the reaction network can induce non-exponential ergodicity when they significantly dominate other reactions across infinitely many regions of the state space.
Interestingly, special architectures allowed us to construct many detailed balanced and complex balanced biochemical systems that are non-exponentially ergodic. Some of these models are low-dimensional bimolecular systems with few reactions. Thus this work suggests the possibility of discovering or synthesizing stochastic systems arising in biochemistry that possess either detailed balancing or complex balancing and slowly converges to their stationary distribution.
\end{abstract}

\classification{60J27, 60J28}

\section{Introduction}
 Continuous-time Markov chains on a discrete state space are commonly used stochastic processes that describe various interacting particle systems. When the state space of a non-explosive continuous-time Markov chain $X$ is finite and irreducible, the time evolution of the probability distribution of the process converges to a limiting distribution $\pi$, a \emph{stationary distribution}, and in this case, $X$ is \emph{ergodic}. Furthermore $X$ on a finite irreducible state space admits \emph{exponential ergodicity}, which means that the convergence of the probability distribution to $\pi$ is exponentially fast in time. 

While finite state spaces are commonly used, especially for interacting particle systems such as the standard Ising models,
countably infinite state spaces often arise in stochastically modeled biochemical systems. For example, some biochemical species in the system can be constantly produced in many existing biochemical systems. In this case, the existence of such a stationary distribution is not always guaranteed, and once it exists, the convergence rate is not necessarily exponential in time. There have been many studies on analytic criteria for ergodicity and exponential ergodicity of continuous-time Markov chains on a countable state space \cite{diaconis2011mathematics, MT-LyaFosterIII,  spectral_ChenMufa1991, tweedie1981criteria}. 
There are few analytic conditions for non-exponential ergodicity such as those proposed in \cite{kingman1964, tweedie1981criteria}.
This is a version of the theorem shown in \cite{kingman1964}. 
\begin{thm} \label{thm: kingman}
    Let $X$ be an ergodic continuous-time Markov chain with state space $\mathbb S$. Suppose that there exists $x\in \mathbb S$ such that for each positive constant $\rho$, 
\begin{align}\label{eq:kingman key equality}
     \mathbb E _x [e^{\rho \tau _ x}] = \infty,   
    \end{align}
     where $\tau_x=\inf\{t> J_x: X(t)=x\}$ is the first return time of $x$, where $J_x=\inf\{t>0: X(t)\neq x\}$. Then $X$ is non-exponentially ergodic.
\end{thm}
\noindent One may expect that showing \eqref{eq:kingman key equality} is challenging due to complicated calculations for exponential moments of the first return time $\tau_x$. In \cite{tweedie1981criteria}, it was shown that solvability of some algebraic equations associated with the transition rates of a given Markov chain is sufficient and necessary for \eqref{eq:kingman key equality}.

In this paper, we provided new sufficient conditions for showing non-exponential ergodicity of continuous-time Markov chains on a countable state space. The conditions, loosely speaking, are met if the state space has a series of paths on which the Markov chain gets trapped. That is, the Markov chain stays long time on the paths with high probability once it enters. 
We believe that the proposed conditions clearly illustrate the behavior of Markov chains that induce non-exponential ergodicity. Also since the path conditions are solely determined by the transition rates of the Markov chains, one can more easily verify the non-exponential ergodicity of a given Markov chain. 

As the main application of this work, we used the path conditions to identify classes of biochemical systems that admit non-exponential ergodicity.
Searching for such classes of biochemical systems endowed with certain dynamical behaviors falls into the study of stochastic chemical reaction network theory, a subfield of probability and mathematical biology.  
 A chemical reaction network is a graphical configuration of the interactions of biochemical species. For instance
 \begin{align*}
A+B\xrightarrow{\kappa_1} C \xrightarrow{\kappa_2} \emptyset.
 \end{align*} The reaction $A+B\to C$ indicates that one copy of species $A$ and one copy of species $B$ dimerize together to produce one copy of species $C$, and the other reaction $C\to \emptyset$ means that one copy of species $C$ is removed from the system. The time evolution of the copy numbers of such chemical species can be stochastically modeled by a continuous-time Markov chain on a countable state space. The constants $\kappa_1$ and $\kappa_2$ are system parameters that will be incorporated in the transition rates for the associated Markov chain.

 When such a reaction network is read as a graph, it is possible that the geometry and topology of the graph directly imply special dynamical behaviors of the associated stochastic dynamical system. The special dynamical behaviors of interest include non-explosivity, irreducibility, ergodicity, and exponential ergodcity. Hence one can classify biochemical  systems using structural properties of the underlying reaction networks so that the classes admit certain dynamical features regardless of system parameters. 
 Ergodicity (positive recurrence) \cite{anderson2020stochastically, anderson2020tier, anderson2018some}, the existence of a unique quasi-stationary distribution \cite{hansen2020existence}, explicit forms of the stationary distributions \cite{AndProdForm, hirono2023squeezing, hong2021derivation}, non-explosivity \cite{anderson2018non} and the large deviation principle \cite{ agazzi2018geometry, agazzi2018large} of the stochastically modeled biochemical systems are studied based on the structural conditions of the reaction networks. Few studies have been done for exponential ergodicity of chemical reaction networks including the case of reaction networks with a single species \cite{XuHansenWiuf2022}, reaction networks having a single connected component \cite{anderson2023mixing}, reaction networks having all homo-dimerizations \cite{anderson2023mixing} and reaction networks having special balancing conditions \cite{anderson2023new}. However, there are no well-known structural conditions guaranteeing non-exponential ergodicity, to the best knowledge.
 
In this paper, we provided classes of reaction networks admitting non-exponential ergodicity. Those classes are characterized by structural conditions identified with our path conditions. Hence the structural conditions also induce the trapping of the associated Markov chains on some paths. More precisely, we showed that if a certain sequence of reactions forms a cycle (so-called a \emph{strong tier-1 cycle}) in the reaction network, then the state space contains infinitely many paths where the associated Markov chain stays for a long time with high probability, and in turn, non-exponential ergodicity can follow. 

Remarkably, we also showed that the proposed structural conditions can be used to find many different non-exponentially ergodic stochastic biochemical systems that are either detailed balanced or complex balanced (a special balancing condition yielding the existence of a Poissonian stationary distribution). 
We also discussed a connection of our work to spectral gap and the congestion ratio that have been typically employed for showing exponential ergodicity. 

This paper is outlined below. In Section \ref{sec:main analytic results}, we introduce the main Markov models and the main analytic conditions for non-exponential ergodicity with motivating examples. In Section \ref{sec:crn}, we introduce reaction networks and the associated Markov chains. Then we define the strong tier-1 cycles and provide the main results about structural conditions for non-exponential ergodicity. Lastly, we talk about the non-exponential ergdocity of detailed balanced and complex balanced reaction systems. All the proofs about the main results in Section \ref{sec:crn} are provided in Section \ref{sec:proofs}. In Section \ref{sec:spectral}, we discuss how the spectral gap and congestion ratios are related to strong tier-1 cycles. In Appendix \ref{app:prove kingman}, we show the proof of Theorem \ref{thm: kingman} for the sake of self-containedness. In Appendix \ref{appdx:prove lemma}, some important lemmas and their proofs are provided.

\section{Main analytic results}\label{sec:main analytic results}
In this section, we introduce the basic Markov models used throughout this paper. We also provide the path conditions that can guarantee non-exponential ergodicity. Before introducing the main models, we introduce some notations that we will use in this paper. 

\subsection{Notations}
For $x\in \mathbb Z^d$, we denote the $i$ th component of $x$ by $(x)_i$. The set of vectors in $\mathbb Z^d$ each of whose component is non-negative is denoted by $\mathbb Z^d_{\ge 0}=\{x \in \mathbb Z^d : (x)_i \ge 0 \text{ for each $i$}\}$. Similarly we define $\mathbb R^d_{\ge 0}=\{x\in \mathbb R^d: (x_i)\ge 0 \text{ for each $i$}\}$. We use $\vec{0}$ to denote the zero vector in $\mathbb Z^d$. As commonly used, $e_i \in \mathbb Z^d$ is the elementary vector whose $i$ th component is $1$ and the other components are $0$.  

For two vectors $u$ and $v$, $u\le v$ and $u<v$ mean that the inequalities hold component-wisely. In other words, $u\le v$ and $u<v$ if and only if for each $i$, $(u)_i\le (v)_i$ and $(u)_i <(v)_i$, respectively. These inequalities may provide a partial order for
multiple vectors. For a sequence $\{a_n\}\subseteq \mathbb Z$, $a_n\equiv c$ means that $a_n=c$ for each $n$.

Let $X=\{X(t)\}_{t\ge 0}$ be a continuous-time Markov chain defined on $\mathbb Z^d_{\ge 0}$.
Throughout this paper, we use $d$ to denote the dimension of a given Markov chain and the number of species for stochastic modeling of reaction networks.
We denote the discrete state space of a continuous-time Markov chain by $\mathbb S$. Especially, if $X(0)=x$ almost surely for some state $x$, then the state space of $X$ can be confined on a subset (not necessarily a proper subset) of $\mathbb Z^d_{\ge 0}$. In this case, we denote the state space by  $\mathbb S_x \subseteq \mathbb Z^d_{\ge 0}$. We denote the rate of the transition from state $z$ to $w$ by $q_{z,w}$ such that
\begin{align*}
   q_{z,w} = \lim_{h\to 0}\frac{P(X(h)=w|X(0)=z)}{h}\ge 0,
\end{align*}
and the total transition rate at $z$ is denoted by $q_z:=\sum_{w\in \mathbb S}q_{z,w}$. We use $P^t(x,\cdot)$ to denote the probability distribution of $X(t)$ initiated at state $x$. That is, for any subset $A$ of $\mathbb S$, $P^t(x,A)=P(X(t)\in A |X(0)=x)$.

\subsection{The basic model}
Let $X=\{X(t)\}_{t\ge 0}$ be a continuous-time Markov chain defined on a discrete state space $\S$. Let $\mathcal A$ be the infinitesimal generator of $X$ that is defined as
\begin{align*}
    \mathcal AV(z)=\sum_{w}q_{z,w}(V(w)-V(z)) \quad \text{for a sutiable function $V:\mathbb S\to \R$.}
\end{align*}
Suppose that $X$ admits a stationary distribution $\pi$ on $\mathbb S$ such that \begin{equation}\label{eq:equilibrium}
\pi(z)\sum_{w\in \mathbb{S}\setminus\{z\}}q_{z,w}=\sum_{w\in \mathbb{S}\setminus\{z\}}\pi(w)q_{w,z}, \quad \text{for each $w\in \mathbb S$.}
\end{equation}
Throughout this paper, assuming irreducibility of $\mathbb S$ and non-explosivity of $X$, we guarantee that for any initial distribution $\mu$ on $\mathbb S$, we have $\displaystyle \lim_{t\to \infty}P_\mu(X(t)\in A)=\pi(A)$ for any subset $A \subseteq \mathbb S$ \cite{NorrisMC97}. %

The convergence rate of $P^t(x,\cdot)$ to $\pi$ is often measured with the total variation norm $\Vert P^t(x,\cdot) - \pi(\cdot) \Vert_{\TV}$, where $\|\mu-\nu\|_{\TV} = \sup_A |\mu(A)-\mu(A)|=\frac12 \sum_{x\in \mathbb{S}} |\mu(x) - \nu(x)|$  for two probability measures on the same discrete measurable space \cite{YuvalLevinMixing}. 
We now define ergodicity and exponential-ergodicity of $
X$.
\begin{defn}
     A continuous-time Markov chain $X$ on a discrete countable state space $\mathbb{S}$ is \emph{ergodic} if $\mathbb S$ is irreducible and there exists a probability measure $\pi$ such that for all $x\in \mathbb S$,
    \begin{align*}
    \norm{P^t(x,\cdot)-\pi(\cdot)}_{TV} \to 0 \quad \text{as} \quad t\to \infty.
    \end{align*}
    $X$ is \emph{exponentially ergodic} if $X$ is ergodic and there exists $\eta>0$ such that for all $x\in \mathbb{S}$, 
    \begin{align}\label{eq:expo ergo}
        \norm{P^t(x,\cdot)-\pi(\cdot)}_{TV} \le B(x)e^{-\eta t}, \quad \text{for all} \ t\ge 0,
    \end{align}
    for some function $B(x):\mathbb{S} \longrightarrow \mathbb{R}_{\ge 0}$. When $X$ is ergodic with a unique stationary distribution $\pi$, if there exists $x\in \mathbb S$ such that  \eqref{eq:expo ergo} does not hold for any constant $B(x)>0$ and $\eta>0$, then $X$ is \emph{non-exponentially ergodic}. 
    \end{defn}

For a continuous-time Markov chain $X$ on the state space $\mathbb S$, `paths' on the state space have been typically used in showing exponential ergodicity (see for example \cite[Chapter 3]{saloff1997lectures} or \cite[Section 3.3]{berestycki2009eight}). We also use paths to describe our sufficient conditions of non-exponential ergodicity. The paths are defined as ordered lists of states. We call that a path is \emph{active} if the transition probabilities between consecutive states are strictly positive. The more precise definitions are as follows.
\begin{defn}
    An ordered list of states (allowing repetition) $\gamma=(x(1),x(2),\dots,x(M)) \subseteq \mathbb S$ is \emph{active} if $q_{x(m),x(m+1)}>0$ for each $m\in \{1,2,\dots,M-1\}$. For a path $\gamma=(x(1),x(2),\dots,x(M))$, the number of states in the path is $M$, which is denoted by $|\gamma|$. For two active paths $\gamma^1=(x(1),\dots,x(|\gamma^1|))$ and $\gamma^2=(z(1),\dots,z(|\gamma^2|))$ with $x(|\gamma^1|)=z(1)$, the composition of the two active paths is  
    $$\gamma^1 \circ \gamma^2:=(x(1),\dots,x(|\gamma^1|),z(2),\dots,z(|\gamma^2|)).$$
\end{defn}
\begin{rem}
     Generally, $\gamma^1 \circ \gamma^2$ and $\gamma^2 \circ \gamma^1$ are different. That is, the order of the composition matters. Also the composition of two active paths is active.
\end{rem}

\begin{defn}
  Let $X$ be a continuous-time Markov chain. A state $z$ is \emph{reachable} from another state $x$ if there exists an active path $\gamma=(x=x(1),x(2),\dots,x(M)=z)$ for some states $x(i)$'s.
\end{defn}

We now define the path probabilities. Let $T_m=\inf\{ t > T_{m-1}: X(t) \neq X(T_{m-1})\}$ be the time for the $m$ th transition of $X$ for $m\in \{1,2,\dots,\}$, where $T_0=0$. 
\begin{defn}
Let $X$ be a continuous-time Markov chain. 
    For an active path $\gamma = (x(1),x(2),\dots,x(M))$, by abusing the notation, we let $\gamma$ also indicate the event such that
    \begin{align*}
        \{X(0) = x(1), X(T_1) = x(2), \dots, X(T_{M-1}) = x(M)\}
    \end{align*}
    Then we define the path probability as
    \begin{align*}
        P(\gamma) = P(X(0) = x(1), X(T_1) = x(2), \dots, X(T_{M-1}) = x(M)).
    \end{align*}
\end{defn}

Let $\gamma=(x(1),x(2),\dots,x(|\gamma|))$ be an active path such that $x(1)=x(|\gamma|)=x$. Conditioned on  $\gamma$, the moment generating function of the returning time $\tau_x$ as in Theorem \ref{thm: kingman} can be explicitly written. For $\rho\in (0,\displaystyle\inf_i q_{x(i)})$,
 by using the strong Markov property, we define the conditioned moment generating function $F(\gamma, \rho)$ as 
\begin{align}
    F(\gamma,\rho)& = \mathbb E_{x}\left [e^{\rho \tau_x} \big | \gamma \right] P(\gamma) =\mathbb E_{x}\left [e^{\rho \tau_x} \big | X(T_m)=x(m), m\in\{1,2,\dots,|\gamma|\} \right] P(\gamma) \label{eq:generate function on paths1} \\
    &=  \prod_{m=1}^{|\gamma|-1} \mathbb E_{x(m)}[e^{\rho T_{x(m)}}] 
 \frac{q_{x(m),x(m+1)}}{q(x_m)}= \prod_{m=1}^{|\gamma|-1} \frac{q_{x(m)}}{q_{x(m
)}-\rho}\frac{q_{x(m),x(m+1)}}{q_{x(m)}} \notag \\
&=\prod_{m=1}^{|\gamma|-1}\frac{q_{x(m),x(m+1)}}{q_{x(m)}-\rho}, \label{eq:generate function on paths2} 
\end{align}
where $T_x$ is the `holding' time at $x$ defined as $T_x=\inf \{t>0 : X(t)\neq x\}$ given $X(0)=x$, which is exponentially distributed with the rate $q_x$.

\subsection{The main theorem}
In this section, the main path condition for non-exponential ergodicity is provided. We begin with a simple motivating example.

\begin{example}\label{ex:the main example 1}
    Let $X$ be a continuous-time Markov chain on $\mathbb Z^2_{\ge 0}$ such that for each $z\in \mathbb Z^2_{\ge 0}$
    \begin{align}\label{eq:main example of ctmc}
        q_{z,z+(1,1)}=1, \ q_{z,z-(1,1)}=(z)_1(z)_2, \ q_{z,z+(0,1)}=(z)_2\text{, and} \ q_{z,z-(0,1)}=(z)_2((z)_2-1).
    \end{align}
    This Markov chain models the copy numbers of species in a biochemical reaction system \eqref{eq:main reaction network example}, which we will precisely define in Section \ref{sec:crn}. Ergodicity of $X$ is held, and its unique stationary distribution $\pi$ is known to be a product form of Possions  as $\pi(x)=\dfrac{e^{-2}}{x_1!x_2!}$ \cite{AndProdForm}. Remarkably, under this stationary distribution, the Markov chain is detailed balance (or reversible) in the sense that $\pi(z)q_{z,w}=\pi(w)q_{w,z}$ for any pair of states $z$ and $w$.

  We found that when $X$ hits state $x_n=(n,0)$ with large enough $n$, $X$ stays in the active path $\gamma_n=(x_n,x_n+(1,1),x_n)$ for a long time due to the following facts:
    \begin{enumerate}
        \item The only possible transition from $x_n$ is the transition to $x'_n:=x_n+(1,1)$ because all the other transition rates are zero at $x_n$. 
         \item While there are three possible transitions at $x'_n$, the dominant transition is the transition to $x_n$ since $q_{x'_n,x_n}=n+1, q_{x'_n,x'_n + (1,1)}=1$, $q_{x'_n,x'_n + (0,1)}=1$, and $q_{x'_n,x'_n - (0,1)}=0$
    \end{enumerate}
This observation can be quantified with $ F(\gamma_n,\rho)$  as
\begin{align}\label{eq:key factor in ex}
  F(\gamma_n,\rho)=\frac{q_{x_n,x'_n}}{q_{x_n}-\rho} \frac{q_{x'_n,x_n}}{q_{x'_n}-\rho}=\frac{1}{1-\rho}\frac{n+1}{n+3-\rho} \quad \text{for small enough $\rho$}.
\end{align}
Then for some fixed large $n$, we have that $F(\gamma_n,\rho)>1$. Now we think of an active path $\Gamma^j_n$ starting at $x=(0,0)$ and returning to $x$ that contains $j$ times of the path $\gamma_n$. Then we will have that for some constant $c>0$
\begin{align*}
    \mathbb E_{x_n}[e^{\rho \tau_{x_n}}] \ge \mathbb E_x \left [e^{\rho \tau_{x_n}}|\Gamma^j_n \right ] P_x\left (\Gamma^j_n \right )\ge c F(\gamma_n,\rho)^j \to \infty, \quad \text{as $j \to \infty$},
\end{align*}
which derives non-exponential ergodicity of $X$ by Theorem \ref{thm: kingman}.
 In other words, the existence of such `trapping' paths on which $X$ spends a long time serves as a sufficient condition for non-exponential ergodicity. We illustrate non-exponential ergodicity of this model in Figure \ref{fig:1}.
  \hfill $\triangle$

 \begin{figure}[!h]
     \centering  \includegraphics[width=0.8\textwidth]{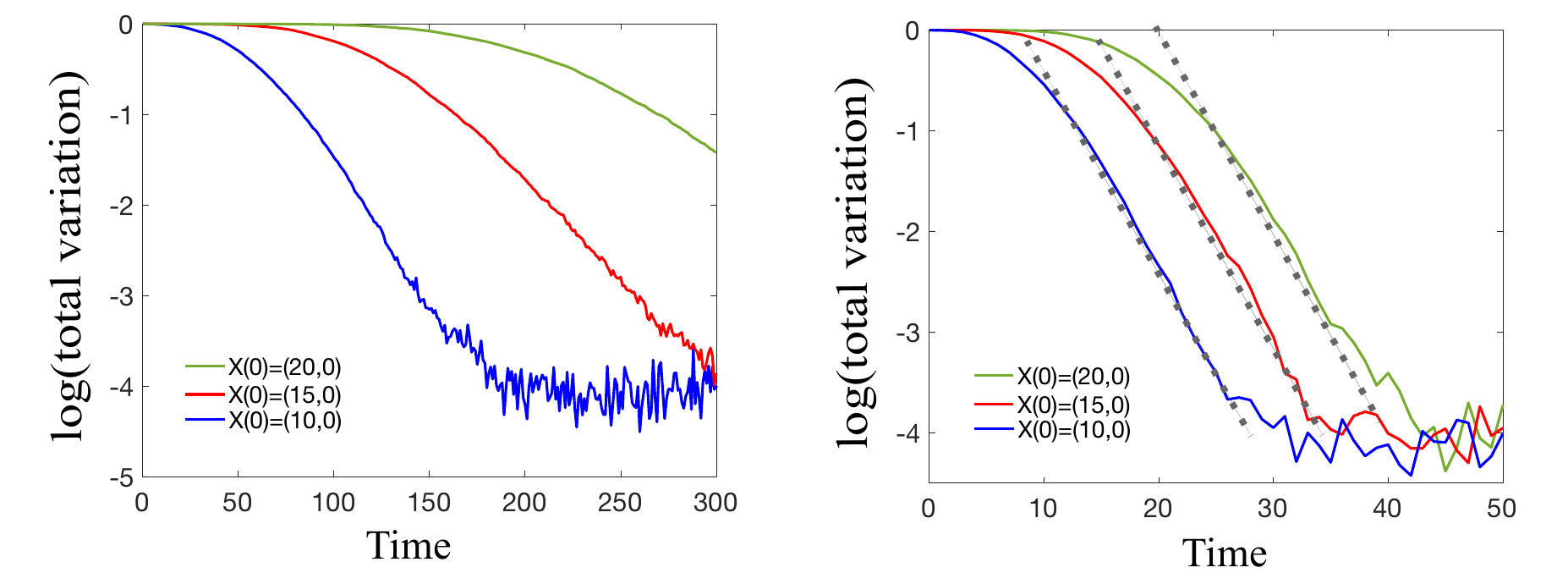}
     \caption{
     Let $X$ be the Markov chains defined in Example \ref{ex:the main example 1}. Let $\bar X$ be the Markov chain with the transition rates $q_{z,z+(1,1)}=1, q_{z,z-(1,1)}=(z)_1(z)_2, q_{z,z+(0,1)}=1$ and $q_{z,z-(0,1)}=(z)_2$ for each state $z$. The graphs of log-scaled total variation norms of $X$ (left) and $\bar X$ (right) are given. Exponentially ergodic of $\bar X$ holds \cite{anderson2023new}. Thus by \eqref{eq:expo ergo} the convergence for the initial states $(10,0), (15,0)$ and $(20,0)$ can be approximated with almost the same straight lines as indicated by the dotted lines. However $X$ is non-exponentially ergodic so that the convergence cannot be approximated with the same straight line for the initial conditions.}
     \label{fig:1}
 \end{figure}

\end{example}
We generalize the ideas shown in Example \ref{ex:the main example 1} with the following theorem.
To avoid trivial cases, we assume that $\inf_{z\in \mathbb S}q_z>0$.
\begin{thm}\label{thm:maybe main}
    Let $X$ be an ergodic continuous-time Markov chain on a countable state space $\mathbb S$. Suppose that there exists $x \in \mathbb S$ satisfying the following condition: for each $\rho\in (0, \displaystyle \inf_{z\in \mathbb S} q_z)$ there exists an active path $\gamma_\rho=(x(1), x(2),\dots,x(M), x(M+1)=x(1)) \subseteq \mathbb S$ with $M\ge 2$ such that
    \begin{enumerate}    
        \item   $x(m) \in \mathbb S$, and $x(m) \neq x$ for all $1 \le m \le M$, and
        \item  $F(\gamma_\rho, \rho)>1 \left (\text{i.e.} \displaystyle \prod_{m=1}^{|\gamma_\rho|-1}\frac{q_{x(m),x(m+1)}}{q_{x(m)}-\rho}>1 \text{ by \eqref{eq:generate function on paths2}}   \right ) $.   
    \end{enumerate}
    Then $X$ is non-exponentially ergodic.
\end{thm}
\begin{proof}
  Let $\rho\in (0, \displaystyle \inf_{z\in \mathbb S} q_z)$ be fixed. Since $X$ is ergodic, $\mathbb S$ is irreducible. Therefore, there exist active paths $\gamma^1_\rho=(x=z(1),z(2),\dots,z(K)=x(1))$  and $\gamma^2_\rho=(x(M+1)=w(1),w(2),\dots,w(L)=x)$ with  some states $z(k)$'s and $w(\ell)$'s that are not equal to $x$. 
    Let $\Gamma^j_\rho=\gamma^1_{\rho} \circ \gamma_\rho \circ \cdots \circ \gamma_\rho \circ \gamma^2_\rho$ be the compsition of the paths where the path $\gamma_{\rho}$ is included $j$ times in the middle. Conditioning on the path $\Gamma^j_\rho$, the returning time $\tau_x$ is equal to $\sum_{k=1}^{K-1} T_{z(k)} + j\sum_{m=1}^{M} T_{x(m)} + \sum_{\ell=1}^{L-1} T_{w(\ell)}$ in distribution, where $T_x$ is the exponential holding time at state $x$.  Then using the strong Markov property we have that 
    \begin{align}
        &\mathbb E _x [ e^{\rho \tau _x} ] \ge \mathbb E _x \left [e^{\rho \tau _x}|\Gamma^j_\rho \right ]P_x\left (\Gamma^j_\rho \right ) \notag \\
        &= \mathbb E _x \left [e^{\rho \sum_{k=1}^{K-1} T_{z(k)}}|\gamma^1_\rho \right ]
         P_x(\gamma^1_\rho) 
        \left ( \mathbb E_{x(1)} \left [e^{\rho \sum_{m=1}^{M} T_{x(m)}}|\gamma_\rho \right ]
         P_{x(1)}(\gamma_\rho) \right )^j  \mathbb E _{x(1)} \left [e^{\rho \sum_{\ell=1}^{L-1} T_{w(\ell)}}|\gamma^2_\rho \right ]
         P_{x(1)}(\gamma^2_\rho). \notag      
    \end{align}
    Then by \eqref{eq:generate function on paths1}
     \begin{align*}
        \mathbb E _x (e^{\rho \tau _x}) &\ge F(\gamma^1_\rho, \rho)F(\gamma_\rho, \rho)^j F(\gamma^2_\rho, \rho),
     \end{align*}
     which approaches to infinity as $j\to \infty$ due to the hypothesis $F(\gamma_\rho,\rho)>1$.
        Therefore, $\mathbb E _x (e^{\rho \tau _x})=\infty$.
        For any $\rho' \ge  \displaystyle \inf_{y\in \mathbb S}q_y$, it is obvious that $\mathbb E_x[e^{\rho' \tau_x}] \ge \mathbb E_x[e^{\rho \tau_x}]$ for any $0<\rho<\displaystyle \inf_{y\in \mathbb S}q_y$. Therefore,
     by Theorem \ref{thm: kingman}, $X$ is non-exponentially ergodic.

\end{proof}

In example \ref{ex:the main example 1}, if we let $x=(0,0)$ and for each $\rho \in (0,\inf_{z\in \mathbb S} q_z)$ if we let $\gamma_\rho=(x_n,x'_n,x_n)$ with large enough $n$, then the conditions in Theorem \ref{thm:maybe main} hold.

\section{Main application: non-exponential ergodicity of stochastic reaction networks}\label{sec:crn}

    As the main application of Theorem \ref{thm:maybe main}, we find classes of stochastic biochemical systems for which the associated Markov chain admit non-exponential ergodicity. The goal of this application is to characterize the classes by structural conditions of the biochemical systems. To do that we first graphically represent the biochemical systems with reaction networks. 
A reaction network is a graph that consists of the set of possible interactions among constituent chemical species. The nodes of the graphs are complexes composed with combinations of chemical species, and the edges of the graphs are reactions between complexes. We can define reaction networks as follows. 
\begin{defn}\label{def:21}
\emph{A  reaction network} is given by a triple of finite sets $(\Sp,\C,\Re)$ such that
\begin{enumerate}
\item  $\Sp=\{S_1,S_2,\cdots,S_d\}$ is a set of $d$ symbols, called the \emph{species} of the network,
\item $\C$ is a set of $y=\sum_{i=1}^d (y)_iS_i$, a linear combinations of $S_i$'s with $(y)_i\in \mathbb Z_{\ge 0}$ for each $i$, called \emph{complexes},
\item $\Re$ is a subset of $\C\times\C$, whose elements $(y,y')$ are typically denoted by $y\to y'$.
\end{enumerate}
\end{defn}
For each $y\to y'\in \Re$, $y$ and $y'$ are called the \emph{source} complex and the \emph{product} complex, respectively. We denote by $\C_s=\{y \in \C: y\to y'\in \Re\}$ the set of the source complexes. For each complex $y\in \C$, we denote by $\Re_y=\{y\to y'\in \Re \text{ for some $y'\in \C$}\}$ the subset of reactions where $y$ serves as the source complex. 

For each $y=\sum_{i=1}^d (y)_iS_i \in \C$, using an abuse of notation, we also denote $y$ by a vector $y=((y)_1,(y)_2,\dots,(y)_d) \in \mathbb Z^d_{\ge 0}$. For example, when $\Sp=\{S_1,S_2,S_3,S_4\}$, then for a complex $y=2S_1+S_2+3S_4$, it can be represented by a vector $y=(2,1,0,3) \in \Z^4_{\ge 0}$. We denote by $y=\emptyset$ for the complex $y$ with $y_i=0$ for each $i$. Commonly we assume that $y\to y\notin\Re$ for all $y\in\C$, and we will do so in the present paper. 

A reaction network $(\Sp,\C,\Re)$ is often illustrated with a directed graph whose nodes are $\C$ and whose edges are $\Re$. For example, 
\begin{equation}\label{ex:crn1}
   \begin{tikzpicture}[baseline={(current bounding box.center)}, scale=0.8]
   \node[state] (1) at (0,2)  {$\emptyset$};
   \node[state] (2) at (2,2)  {$A$};
   \node[state] (3) at (5,2)  {$B+C$};
   \node[state] (4) at (7,2)  {$2A$};
   \node[state] (5) at (9,2)  {$C$};
   \node[state] (6) at (11,2)  {$B$};
   \node[state] (7) at (14,2)  {$A+B+C$};
   \node[state] (8) at (12.5,0)  {$D$};
   \path[->]
    (2) edge node {} (3)
    (1) edge node {} (2)
    (1) edge[bend right=25] node {} (3)
    (4) edge[bend left] node {} (5)
    (5) edge[bend left] node {} (4)
    (6) edge node {} (7)
    (7) edge node {} (8)
    (8) edge node {} (6);
  \end{tikzpicture}
\end{equation}

We introduce some terminologies that are generally used to describe some structural features of reaction networks.
\begin{defn}\label{def:WR}
Let $(\Sp,\C,\Re)$ be a reaction network.  We say that a connected component in the reaction network is \emph{weakly reversible} if it is strongly connected, i.e., if for any $y\to y' \in \Re$ there exist complexes $y_i$ such that $( y_1\to  y_2, y_2\to  y_3,\dots,  y_{M-1}\to  y_{M},  y_{M}\to  y_1) \subseteq \Re$, where $ y_1=y$ and $ y_2=y'$.
If all connected components in a reaction network are weakly reversible, then the reaction network is said to be weakly reversible. 
\end{defn}

\begin{example}\label{ex11}
  The reaction network in \eqref{ex:crn1} consists of three connected components. While the right-most and middle connected components are weakly reversible, the first one is not. Due to the presence of non-weakly reversible components, the entire network is not weakly reversible.
\hfill $\triangle$
\end{example}

\subsection{Stochastic models for reaction networks}
\label{sec:dynamics}

In this section we introduce a stochastic process associated with reaction networks, which is often utilized in biochemistry when the abundances of species are low. Assuming the typical spatially well-mixed status of the chemical systems, the stochastic dynamics will be defined as a continuous-time Markov chain on $\mathbb Z^d_{\ge 0}$ to model the copy numbers of species in the biochemical systems. 

First, each reaction $y\to y'$ is associated with an \emph{intensity function} $\lambda_{y\to y'}\colon \Z^d\to \R_{\geq0}$, which quantifies the expected time for the reaction to occur. It is natural to assume that $\lambda_{y\to y'}(x)>0$ only if $x\geq y$. This means that $y\to y'$ can take place only if enough reacting species exist.  \emph{Mass-action kinetics} is most common kinetics for defining $\lambda_{y\to y'}$, where the intensity of a reaction is proportional to the number of combinations of the present species engaged to the reaction. Specifically, $\lambda_{y\to y'}$'s are defined under mass-action kinetic as
\begin{align}\label{mass}
\lambda_{y\rightarrow y'}(x)= \kappa_{y\rightarrow y'} \lambda_y(x), \quad\text{where}\quad \lambda_y(x)=\prod_{i=1}^d \frac{x_i !}{(x_i-y_{i})!}\mathbf{1}_{\{x_i \ge y_{i}\}},
\end{align}
for some positive parameters $\kappa_{y\to y'}$.   A reaction network together with a choice of a set of parameters $\kappa=\{\kappa_{y\to y'}:y\to y'\in \Re\}$ is denoted by $(\Sp, \C, \Re, \kappa)$ and is called a \emph{reaction system}.  Once $\kappa$ is incorporated in a reaction network, we denote each reaction $y\to y'\in \Re$ by $y\xrightarrow{\kappa_{y\to y'}} y'$ when the reaction network is graphically represented.
Throughout this paper, we assume mass-action for the stochastic models of reaction networks.

For a reaction system $(\Sp, \C, \Re, \kappa)$, let $X=\{X(t)\}_{t\ge 0}$ be the associated continuous-time Markov chain. Then the transition rate from state $z$ to state $w$ are given by
\begin{equation}\label{eq:q}
    q_{z,w}=\sum_{\substack{y\to y' \in \Re:\\ y'-y=w-z}}\lambda_{y\to y'}(z).
\end{equation}
Then for every $t\ge 0$, $(X(t))_i$ represents the copy number of species $S_i$ at time $t$. That is, the transition of $X$  generated by a reaction $y\to y'$ yields the change of the copy numbers of species.  
The generator $\mathcal{A}$ of $X$ is given by
\begin{align}\label{gen5}
 \mathcal{A}V(z) = \sum_{z}q_{z,w}(V(w)-V(z) )=\sum_{y \rightarrow y' \in \Re} \lambda_{y\rightarrow y'}(z)(V(z+y'-y)-V(z)), 
\end{align}
where $V:\Z^d_{\ge 0} \to \R$ \cite{Kurtz86}.
The state space of $X$ with $X(0)=x$ is denoted by $\mathbb S_x$, as mentioned before. Note that  $q_{z,w}>0$ for $z,w \in \mathbb Z^d_{\ge 0}$ if and only if there exists $y\to y'\in \Re$ such that $w-z=y\to y'$ and $\lambda_{y\to y'}(z)=\kappa_{y\to y'}\lambda_y(z)>0$ (equivalently $z>y$). Thus the state space $\mathbb S_x$ is determined by  $\Re$ and the initial stat $x$ independently from the choice of the parameters $\kappa$.

\subsection{Strong tier-1 cycles}\label{sec:strong}

In this section, we provide the key structural features for showing non-exponential ergodicity when the Markov chain associated with a reaction network is ergodic under a choice of the parameters $\kappa$. The main idea  is to find a set of reactions that generate paths satisfying the conditions of Theorem \ref{thm:maybe main}. The following example will provide an insight about the connection between the structural conditions of reaction networks and the path conditions of Theorem \ref{thm:maybe main}. For the sake of readability, we provide the proofs of Theorem \ref{thm:maybe main} and Proposition \ref{prop:strong tier 1} of this section later in Section \ref{sec:proofs}.  

\begin{example}\label{ex:strong tier 1}
    Consider the following reaction system $(\Sp, \C, \Re, \kappa)$.
\begin{align}\label{eq:main reaction network example}
    \emptyset \xrightleftharpoons[\kappa_2]{\kappa_1} A+B, \quad B \xrightleftharpoons[\kappa_4]{\kappa_3} 2B.
\end{align}
The associated Markov chain $X$ for the reaction network \eqref{eq:main reaction network example} under mass-action kinetics \eqref{mass} is the same as the continuous-time Markov chain introduced in Example \ref{ex:the main example 1} if we set $\kappa_i = 1$ for each $i$. Similarly to Example \ref{ex:the main example 1},
for any choice of $\kappa_i$'s, we can show that for large enough $n$, $X(t)$ stays a long time in the path $\gamma_n=((n,0), (n+1,1), (n,0))$ with high probability as described in Figure \ref{fig:2}. Interestingly by reactions $\emptyset \rightleftharpoons A+B$ forming a cycle in the reaction network \eqref{eq:main reaction network example}, $X$ transitions in between the states over $\gamma_n$. To find general structural conditions from this observation, we interpret the main condition $F(\gamma_n,\rho)>1$ for  each $0<\rho<\displaystyle\min_i \kappa_i = \displaystyle\inf_{z\in \mathbb Z^2_{\ge 0}} q_z$ in Theorem \ref{thm:maybe main} as an outcome of the relationship between the intensities of the reactions. To do so we first note that by \eqref{eq:generate function on paths2} 
\begin{align}
    F(\gamma_n,\rho) &= \frac{\lambda_{\emptyset \to A+B}(n,0)}{\lambda_{\emptyset \to A+B}(n,0)-\rho} \frac{\lambda_{A+B \to \emptyset} (n+1,1)}{\lambda_{A+B \to \emptyset} (n+1,1) + \lambda_{\emptyset \to A+B}(n+1,1) + \lambda_{B \to 2B}(n+1,1)-\rho} \notag\\
    \begin{split}\label{eq:how to bound F}
    &= \left (1 + \frac{\rho}{\lambda_{\emptyset \to A+B}(n,0)-\rho} \right )  \\ 
    & \times 
 \left( 1 - \frac{\lambda_{\emptyset \to A+B}(n+1,1)+\lambda_{B \to 2B}(n+1,1)-\rho}{\lambda_{A+B \to \emptyset} (n+1,1) + \lambda_{\emptyset \to A+B}(n+1,1) + \lambda_{B \to 2B}(n+1,1)-\rho}\right)
 \end{split}
\end{align}
Since $\lambda_{A+B \to \emptyset} (n+1,1) + \lambda_{\emptyset \to A+B}(n+1,1) + \lambda_{B \to 2B}(n+1,1)-\rho \ge  \lambda_{A+B \to \emptyset} (n+1,1)$ for each $n$,
  \eqref{eq:how to bound F} implies that
\begin{align}
    F(\gamma_n,\rho) &\ge \left( 1+\frac{\rho}{\lambda_{\emptyset\to A+B}(n,0)} \right )\left( 1 - \frac{\lambda_{\emptyset \to A+B}(n+1,1)+\lambda_{B \to 2B}(n+1,1)}{\lambda_{A+B \to \emptyset} (n+1,1) }\right) \label{eq:how to bound F2}.
\end{align}
Consequently the inequality $F(\gamma_n,\rho)>1$ is up to how quickly the terms $\dfrac{\lambda_{\emptyset\to A+B}(n+1,1)}{\lambda_{A+B\to \emptyset} (n+1,1) }$ and $\dfrac{\lambda_{B\to 2B}(n+1,1)}{\lambda_{A+B\to \emptyset} (n+1,1) }$ go to $0$ in comparison to $\dfrac{1}{\lambda_{\emptyset\to A+B}(n,0)}$, as $n\to \infty$. This comparison will be the key to generalize the observation of this example. 

In fact, both $\left(\dfrac{\lambda_{\emptyset\to A+B}(n+1,1)}{\lambda_{A+B\to \emptyset}(n+1,1)}\right ) / \left (\dfrac{1}{\lambda_{\emptyset\to A+B}(n,0)} \right )$ and $\left(\dfrac{\lambda_{B\to 2B}(n+1,1)}{\lambda_{A+B\to \emptyset}(n+1,1)}\right ) / \left (\dfrac{1}{\lambda_{\emptyset\to A+B}(n,0)} \right )$  approach to $0$, as $n\to \infty$. Using this we have  that for a sufficiently large $n$,
\begin{align*}
    F(\gamma_n,\rho) \ge 1+\frac{\rho}{\lambda_{\emptyset\to A+B}(n,0)}(1-\epsilon) > 1 \quad \text{for some sufficiently small $\epsilon>0$}
\end{align*}
because by \eqref{eq:how to bound F2}
\begin{align*}
F(\gamma_n,\rho)&\ge1+\frac{\rho}{\lambda_{\emptyset\to A+B}(n,0)}\left(1-\frac{\lambda_{\emptyset\to A+B}(n,0)\lambda_{\emptyset \to A+B}(n+1,1)}{\rho \lambda_{A+B \to \emptyset}(n+1,1)}-\frac{\lambda_{\emptyset\to A+B}(n,0)\lambda_{B \to 2B}(n+1,1)}{\rho \lambda_{A+B \to \emptyset}(n+1,1)}\right. \\
    & \hspace{4cm} \left. -\frac{\lambda_{\emptyset \to A+B}(n+1,1)+\lambda_{B \to 2B}(n+1,1)}{\lambda_{A+B \to \emptyset} (n+1,1) }\right).
\end{align*} 
 \hfill $\triangle$
\begin{figure}[!h]
    \centering
\includegraphics[width=0.8\textwidth]{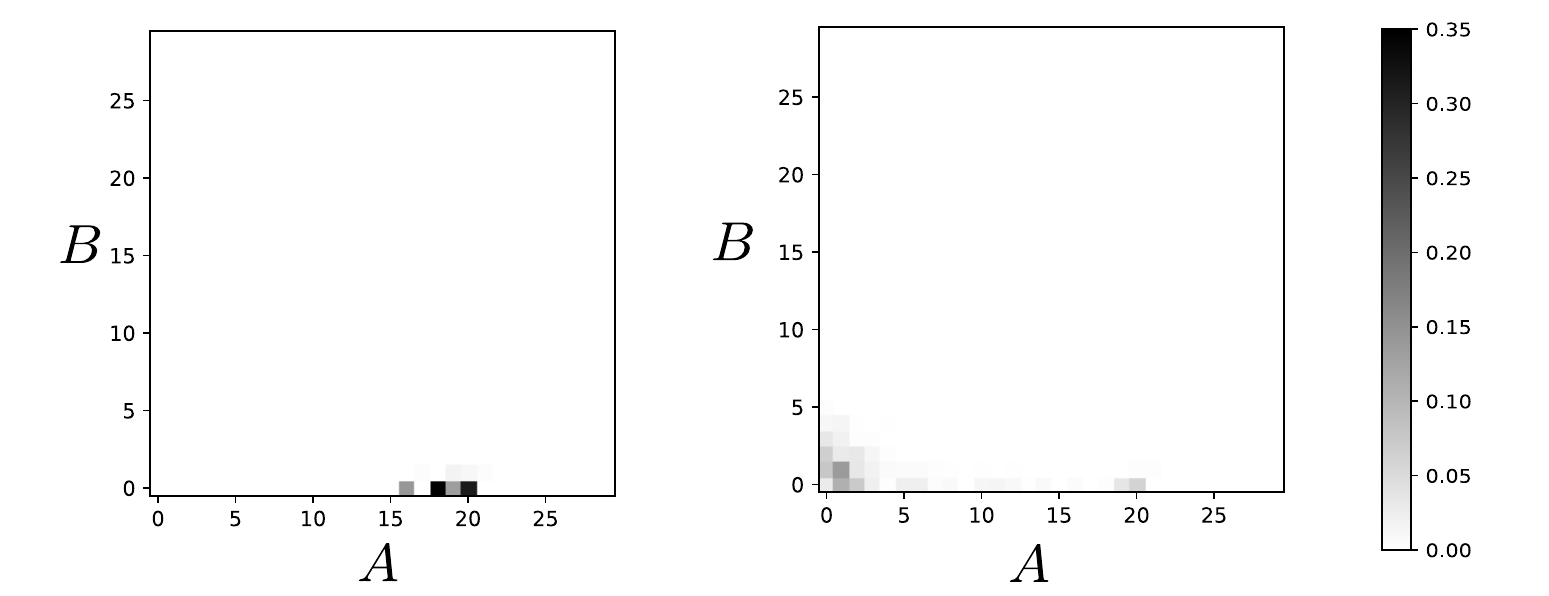}
    \caption{Heat maps of trajectories of the associated Markov chains for \eqref{eq:main reaction network example} (left) and another reaction network $B\rightleftharpoons \emptyset \rightleftharpoons A+B$ (right), respectively within a time interval $[0,100]$. The exponential ergodicity of the second model can be shown in the way provided in \cite{anderson2023new}. The trajectory of the non-exponentially ergodic model \eqref{eq:main reaction network example} is trapped on the path $((n,0),(n+1,1),(n,0))$ while the second model moves towards the origin quickly.}
    \label{fig:2}
\end{figure}  
 \end{example}

 In Example \ref{ex:strong tier 1}, it was shown that 
the ratios of the reaction intensities are the key conditions. To generalize this we define a special subset of reactions that can induce non-exponential ergodicity. 
\begin{defn}\label{def:strong tier1}
   Let $(\Sp,\C,\Re)$ be a reaction network. Let $R=(y^*_1\rightarrow y^*_2, y^*_2\rightarrow y^*_3, \cdots, y^*_M \rightarrow y^*_1 ) \subseteq \Re$ be an ordered list of reactions in $\Re$. We call $R$ a \text{\emph{strong tier-1 cycle}} along a sequence $\{x_n\} \subseteq \mathbb Z^d_{\ge 0}$ if
   \begin{enumerate}
      \item $\lambda_{y^*_1}(x_n)>0$ for all $n$, and 
       \item there exists $m_0$ such that for each $m\in \{1,\dots,M\}$ we have
        \begin{align}\label{eq:key for strong tier 1}
            \lim _{n\rightarrow \infty} \left ( \dfrac{\lambda_{y^*_{m_0}}(x^{m_0}_n)\lambda_{y}(x^m_n)}{\lambda_{y^*_m}(x^m_n)} \right ) = 0,
        \end{align}
        whenever there exists $y\to y'\in \Re\setminus \{y^*_m\to y^*_{m+1}\}$ (regarding $y^*_{M+1}=y^*_1$), where $x^m_n:=x_n+\sum_{j=1}^{m-1} y^*_{j+1}-y^*_j$ for each $m=2,\dots,M$ and $x^1_n=x_n$.
   \end{enumerate}
\end{defn}

\begin{rem} 
In \eqref{eq:key for strong tier 1}, the term $\dfrac{\lambda_{y}(x^m_n)}{\lambda_{y^*_m}(x^m_n)}$ quantifies how dominant $y^*_m\to y^*_{m+1}$ is at $x^m_n$. The reciprocal of the term $\lambda_{y^*_{m_0}}(x^{m_0}_n)$ is the expected holding time at $x^{m_0}_n$. Hence \eqref{eq:key for strong tier 1} relates the probability and the time for the associated Markov chain to stay in the paths generated by the reactions in $R$. This is the key condition for non-exponential ergodicity as shown in Example \ref{ex:strong tier 1}.
\end{rem}

\begin{rem} \label{rem:why cycle}
We constructed a strong tier-1 cycle $R$ with reactions forming a cycle in the reaction network in order to generate an active path by using the reactions in $R$. More precisely, we let $R=(y_1\to y'_1,y_2\to y'_2,\dots,y_M\to y'_M)$ be an ordered list of reactions, and for some $x\in \mathbb Z^d_{\ge 0}$ we define $x(m)=x+\sum_{j=1}^{m-1}(y'_j-y_j)$ for each $m$. Then a path $\gamma=(x(1),\dots,x(M))$ generated by $R$ is active if $R$ is a cycle (i.e. $y_{m+1}=y'_m$) and $x>y_1$.  This is because if a reaction $y\to y'$ is available at state $x$ (i.e. $x>y$), then $y'\to y''$ is available at $x+y'-y$ since $x+y'-y>y'$.
\end{rem}

The term `tier' is derived from the concept of \emph{tiers} in the study of chemical reaction systems. It provides a special language used to compare the size of the reactions intensities. The tiers were originally invented in \cite{AndGAC_ONE2011} in order to study the long-term behaviors of deterministic reaction systems and have been recently used to show ergodicity and exponential ergodicity of some classes of stochastic reaction systems \cite{AndGAC_ONE2011, anderson2020stochastically,  anderson2020tier, anderson2018some, anderson2023mixing}. Although the properties of strong tier-1 cycles are closely related to tiers, we do not explicitly use the tiers for the main methods of this paper.

As one may expect, the reaction $y^*_m\to y^*_{m+1}$ in a strong tier-1 cycle is a dominant reaction at $x^m_n$. However, a strong tier-1 cycle has additional features. For a given strong tier-1 cycle $R=(y^*_1\rightarrow y^*_2, y^*_2\rightarrow y^*_3, \cdots, y^*_M \rightarrow y^*_1 )$, we keep using the notation $x^m_n:=x_n+\sum_{j=1}^{m-1} y^*_{j+1}-y^*_j$ for each $m=2,\dots,M$ and $x^1_n=x_n$ throughout this paper.
\begin{prop}\label{prop:strong tier 1}
    For a reaction network $(\Sp,\C,\Re)$, suppose that $R=(y^*_1\rightarrow y^*_2, y^*_2\rightarrow y^*_3, \cdots, y^*_M \rightarrow y^*_1 )$ is a strong tier-1 cycle along $\{x_n\}$. 
    Then, considering a subsequence of $\{x_n\}$ if needed, we have that
    \begin{enumerate}
        \item for each $m$, $\lambda_{y^*_m}(x^m_n)>0$ for all $n$ and $\dlim_{n\to \infty} \frac{\lambda_{y}(x^m_n)}{\lambda_{y^*_m}(x^m_n)}=0$ for any $y\in \C_s\setminus \{y^*_m\}$,  
        \item $\Re_{y^*_m}:=\{y^*_m \to y' \in \Re\}=\{y^*_m\to y^*_{m+1}\}$ for each $m\in \{1,\dots,M\}$ (regarding $y^*_{M+1}=y^*_1$),
       \item for each $n$, $\lambda_y(x^{m_0}_n)=0$ if $y\in \C_s \setminus \{y^*_{m_0}\}$, and
       \item $\{i: \displaystyle\sup_{n}(x_n)_i< \infty \}\neq \emptyset$.
    \end{enumerate}
\end{prop}
\begin{example}
For \eqref{eq:main reaction network example}, let $y^*_1=\emptyset$ and $y^*_2=A+B$ . Then $R=(y^*_1 \to y^*_2, y^*_2\to y^*_1)$ is a strong tier-1 cycle along $\{(n,0)\}$ with $m_0=1$. 
As stated in Proposition \ref{prop:strong tier 1}, we can also see that 
\begin{align*}
    &\lim_{n\to \infty}\frac{\lambda_y(n,0)}{\lambda_{\emptyset}(n,0)}=0 \quad \text{for any $y\in \C_s\setminus \{\emptyset\}$ and } \\
    &\lim_{n\to \infty}\frac{\lambda_y(n+1,1)}{\lambda_{A+B}(n+1,1)}=0 \quad \text{for any $y\in \C_s\setminus \{A+B\}$}.
\end{align*}
Also $\lambda_{y}(n,0)\equiv 0$ for any $y\in \C_s\setminus\{\emptyset\}$. \hfill $\triangle$
\end{example}

For a strong tier-1 cycle along $\{x_n\}$, the  subsequence considered in Proposition \ref{prop:strong tier 1} will simplify the arguments for our main results in the following sections. Since the condition \eqref{eq:key for strong tier 1} still holds for any subsequence of $\{x_n\}$, we will consider only a special type of sequences for strong tier-1 cycles from now on. 
\begin{defn}\label{def_tier-seqence}
Let $(\Sp,\C,\Re)$ be a reaction network. An infinite sequence $\{x_n\} \subseteq \mathbb Z^d_{\ge 0}$ is called a \emph{tier sequence} of $(\Sp,\C,\Re)$ if
\begin{enumerate}
\item $\dlim_{n\rightarrow\infty}(x_n)_i \in [0,\infty]$ for each $i$, and $I_{\{x_n\}}:=\{i : \dlim_{n\to \infty}(x_n)_i=\infty\}$ is not empty,
\item for any $n$, $(x_n)_i \le (x_{n+1})_i$ if $i\in I_{\{x_n\}}$ and  $(x_n)_i \equiv c_i$ for some $c_i\in \mathbb Z_{\ge 0}$ otherwise, and
\item for any two complexes $y,y' \in \C_s$, the limit $\displaystyle \lim_{n\to \infty}\frac{\lambda_y(x_n)}{\lambda_{y'}(x_n)}$ exists in $[0,\infty]$ (i.e. the reaction intensities are mutually comparable along $\{x_n\}$).
\end{enumerate}
\end{defn}
Tier-sequences are used in studying qualitative behaviors of stochastic reaction systems with tiers \cite{anderson2020stochastically, anderson2020tier, anderson2018some, anderson2023mixing} although the definition in this paper is slightly different from that of the previously works. 
\begin{rem}
    If $R$ is a strong tier-1 cycle along a tier sequence $\{x_n\}$, then the results in Proposition \ref{prop:strong tier 1} follow without considering a further subsequence.
\end{rem}

Now we show that the existence of a strong tier-1 cycle $R=\{y^*_1\to y^*_2,\dots, y^*_M\to y^*_1\}$ along $\{x_n\}$ implies non-exponential ergodicity of the associated Markov chain $X$ on the state space $\mathbb S$ if $X$ is ergodic and $\{x_n\}$ lying in the state space. Similarly to the derivation shown in Example \ref{ex:strong tier 1}, the main idea underlying the theorem below is to show the conditions in Theorem \ref{thm:maybe main}.

\begin{thm}\label{thm: stong tier1 and irr is enough}
    Let $(\Sp,\C,\Re)$ be a reaction network.
     We assume that 
    $(\Sp,\C,\Re)$ has a \text{\emph{strong tier-1 cycle}} $R=\{y^*_1\rightarrow y^*_2, y^*_2\rightarrow y^*_3, \cdots, y^*_M \rightarrow y^*_{M+1} \} \subseteq \Re$ along a tier sequence $\{x_n\}$ such that $\{x_n\} \subseteq \mathbb S_x$.
    If under a choice of the parameters $\kappa$, the associated Markov chain $X$ for $(\Sp,\C,\Re,\kappa)$ 
 with $X(0)=x$ is ergodic on the state space $\mathbb S_x$, then $X$ is non-exponentially ergodic.  
\end{thm}
The assumption $\{x_n\} \subseteq \mathbb S_x$ in Theorem \ref{thm: stong tier1 and irr is enough} will be essential for non-exponential ergodicity (in Appendix \ref{appdx:prove lemma} we provide simple examples to highlight the importance of the condition $\{x_n\} \subseteq \mathbb S_x$). When a strong tier-1 sequence exists along $\{x_n\}$ such that $\{x_n\}$ is not lying on $\mathbb S_x$, however, if weak reversibility (Definition \ref{def:WR}) and an additional reachability condition holds, we can use $\{x_n\}$ to construct a new tier-sequence $\{\bar x_n\} \subseteq \mathbb S_x$ along which the same $R$ is a strong tier-1 cycle. Therefore non-exponential ergodicity still follows. We provide more details about this in Appendix \ref{appdx:prove lemma}.

\subsection{Structural conditions for non-exponential ergodicity}\label{subsec:structural}
Now we introduce our main results about classes of reaction networks that contain a strong tier-1 cycle and eventually admit non-exponential ergodicity. The key idea for following results is that totally ordered complexes by $<$ can form a strong tier-1 cycle. For the sake of readability, we also provide the proofs of Corollary \ref{cor:structure1}--\ref{cor:structure4} of this section later in Section \ref{sec:proofs}.  
 \begin{cor}\label{cor:structure1}
   Let $(\Sp,\C,\Re)$ be a weakly reversible reaction network such that there exists a subset of reactions of the form $\{y^*_1\rightarrow y^*_2, y^*_2\rightarrow y^*_3, \cdots, y^*_{M-1}\to y^*_{M}, y^*_M \rightarrow y^*_1 \} \subseteq \Re$.  Suppose that 
    \begin{enumerate}
        \item $\C$ is totally ordered by $<$. That is, for any pair of $y, y' \in \C$, either $y < y'$ or $y' < y$,
        \item $\Re_{y^*_m} = \{y^*_m \to y^*_{m+1}\}$ for each $m \in \{1,2,\dots, M\}$, 
        \item $y^*_{m_0}=\emptyset$ for some $m_0$, and
        \item for the elementary vectors $e_i$'s, we have $U:=\{i: e_i \in \text{span}\{y'-y : y\to y'\in \Re\}\} \neq \emptyset$.
    \end{enumerate}

    Let $x=y^*_{m'}$ for some $m'$. If under a choice of the parameters $\kappa$, the associated Markov chain $X$ for $(\Sp,\C,\Re,\kappa)$ 
 with $X(0)=x$ is ergodic on $\mathbb S_x$, then $X$ is non-exponentially ergodic.
\end{cor}

\begin{example}\label{ex:structural1}
    Consider the following reaction network $(\Sp, \C, \Re)$.
\begin{align}
\begin{split}
    &\emptyset \ \ \longrightarrow \ \ 4A+4B+3C \\[-0.5ex]
    &  \quad \displaystyle \nwarrow \hspace{1.5cm}  \swarrow  \qquad \qquad 2A+2B+2C \rightleftharpoons 6A+5B+4C. \\[-0.5ex]
& \hspace{0.6cm} A+B+C \hspace{1.2cm} 
\end{split}
\end{align}  
It is known that  the associated Markov chain $X$ chains is ergodic under any choice of parameters by \cite[Theorem 4.2]{AndProdForm} (more details about this are provided in Section \ref{sec:detail complex}).  
Now we show that assumptions 1--4 of Corollary \ref{cor:structure1} hold for $(\Sp, \C, \Re)$. First of all, it is clear that $\C$ is totally ordered by $<$. Denote $y^*_1 = \emptyset$, $y^*_2 = 4A +4B +3C$, and $y^*_3 = A + B +C$.
For $R = (y^*_1 \to y^*_2, y^*_2 \to y^*_3, y^*_3 \to y^*_1)$,
assumptions 2 and 3 hold. Furthermore, $\text{span}\{y\to y': y\to y'\in \Re\}=\mathbb R^3$, which guarantees assumption 4. Hence, $X$ with $X(0)=\vec{0}$ is non-exponentially ergodic by Corollary \ref{cor:structure1}.
\hfill $\triangle$
\end{example}

For reaction networks consisting of only two species, we can alternate the conditions in Corollary \ref{cor:structure1}.
\begin{cor}\label{cor:structure2}
    Let $(\Sp,\C,\Re)$ be a reaction network such that $\Sp = \{S_1, S_2\}$, 
   $\C = \{y^*_1, y^*_2, \dots, y^*_M\}$, and $\Re = \{y^*_1\rightarrow y^*_2, y^*_2\rightarrow y^*_3, \cdots, y^*_{M-1}\to y^*_{M}, y^*_M \rightarrow y^*_1 \}$. Suppose that
    \begin{enumerate}
        \item $\C$ can be written as $\C=\{\bar y_1,\bar y_2,\dots,\bar y_M\}$ with $\bar y_m < \bar y_{m+1}$ for each $m=1,2,\dots,M-1$ (i.e. $\C$ is totally ordered by $<$),  
        \item there exists $i_1\in \{1,2\}$ such that $(\bar y_k)_{i_1}<(\bar y_{k'})_{i_1}-(\bar y_1)_{i_1}$ for any $1\le m<m'\le M$, and  
        \item $\text{dim}(\text{span}\{y'-y : y\to y'\in \Re\}) =2$.
    \end{enumerate}
    Let $x=y^*_{m'}$ for some $m'$. If under a choice of the parameters $\kappa$, the associated Markov chain $X$ for $(\Sp,\C,\Re,\kappa)$ 
 with $X(0)=x$ is ergodic on $\mathbb S_x$, then $X$ is non-exponentially ergodic.
\end{cor}

\begin{example}\label{ex: structure 2}
    Consider the following reaction network $(\Sp, \C, \Re)$.
  \begin{align}
\begin{split}
    &A+B \ \ \longrightarrow \ \ 6A+3B. \\[-0.5ex]
    &  \quad \displaystyle \nwarrow \hspace{1.5cm}  \swarrow  \\[-0.5ex]
& \hspace{1cm} 3A+2B \hspace{1.2cm} 
\end{split}
\end{align}

By the same argument in Example \ref{ex:structural1}, the associated Markov chain $X$ is ergodic under any choice of parameters. Now we show that assumptions 1--3 of Corollary \ref{cor:structure2} hold for $(\Sp, \C, \Re)$.  
First of all, denote $S_1 = A$, $S_2 = B$, $y^*_1 = A+B$, $y^*_2 = 6A + 3B$, and $y^*_3 = 3A+2B$. Then those complexes are ordered as $\bar y_1 = y^*_1$, $\bar y_2 = y^*_3$, and $\bar y_3 = y^*_2$ so that condition 1 holds.
Secondly, with $i_1 = 1$ and $i_2=2$, condition 2  follows. Obviously condition 3 holds. Therefore, if $X(0) = (1,1)$, then $X$ is non-exponentially ergodic by Corollary \ref{cor:structure2}.     
\hfill $\triangle$
\end{example}

We slightly generalize the conditions of Corollary \ref{cor:structure2} allowing additional reactions whose intensities are weak enough (even equal to $0$) along $\{x_n\}$.
The function $f$, defined in the following corollary, offers a criterion for determining if a reaction is sufficiently weak.

\begin{cor}\label{cor:structure3}
Let $(\Sp,\C,\Re)$ be a weakly reversible reaction network such that $\Sp = \{S_1, S_2\}$, 
    and $\{y^*_1\rightarrow y^*_2, y^*_2\rightarrow y^*_3, \cdots, y^*_{M-1}\to y^*_{M}, y^*_M \rightarrow y^*_1 \} \subseteq \Re$. Suppose that
    \begin{enumerate}
        \item $\{y^*_1,y^*_2,\dots,y^*_M\}$ can be written as $\{\bar y_1,\bar y_2,\dots,\bar y_M\}$ with $\bar y_m < \bar y_{m+1}$ for each $m=1,2,\dots,M-1$,  
        \item there exists $i_1\in \{1,2\}$ such that $(\bar y_k)_{i_1}<(\bar y_{k'})_{i_1}-(\bar y_1)_{i_1}$ for any $1\le m<m'\le M$,   
        \item $\text{dim}(\text{span}\{y'-y : y\to y'\in \Re\}) =2$,
        \item for each $m \in \{1,2,\dots,M\}$, $\Re_{y^*_m} = \{y^*_m \to y^*_{m+1}\}$, and
        \item letting $i_2\in \{1,2\}$ such that $i_2\neq i_1$,  we further assume that $ f((y)_{i_1} + (\bar y_1)_{i_1})<(y)_{i_2}$ for each $y \in \C \setminus \{\bar y_1\dots,\bar y_M\}$, where
    $f: \Z_{\ge 0} \to \Z_{\ge 0}$ is a function such that 
 \begin{align}
     f(\ell) =\begin{cases}\label{eq:function f}
         0 \quad &\text{if $\ell \in [0,(\bar y_1)_{i_1})$}\\
         (\bar y_m)_{i_2} \quad &\text{if $\ell\in [(\bar y_m)_{i_1}, (\bar y_{m+1})_{i_1}) $ for some $m\in \{1,2,\dots,M-1\}$}\\
         (\bar y_M)_{i_2} &\text{if $\ell\in [(\bar y_M)_{i_1},\infty)$}.
     \end{cases}
 \end{align} 
    \end{enumerate}
    Let $x=y^*_{m'}$ for some $m'$. If under a choice of the parameters $\kappa$, the associated Markov chain $X$ for $(\Sp,\C,\Re,\kappa)$ 
 with $X(0)=x$ is ergodic on $\mathbb S_x$, then $X$ is non-exponentially ergodic.
\end{cor} 

We visualize the location of the complexes $y\in \C\setminus\{\bar y_1,\dots,\bar y_M\}$ in Figure \ref{fig:3} by viewing $y$ as a vector in $\mathbb Z^d_{\ge 0}$
\begin{figure}[!h]
    \centering
    \includegraphics[width=0.8\textwidth]{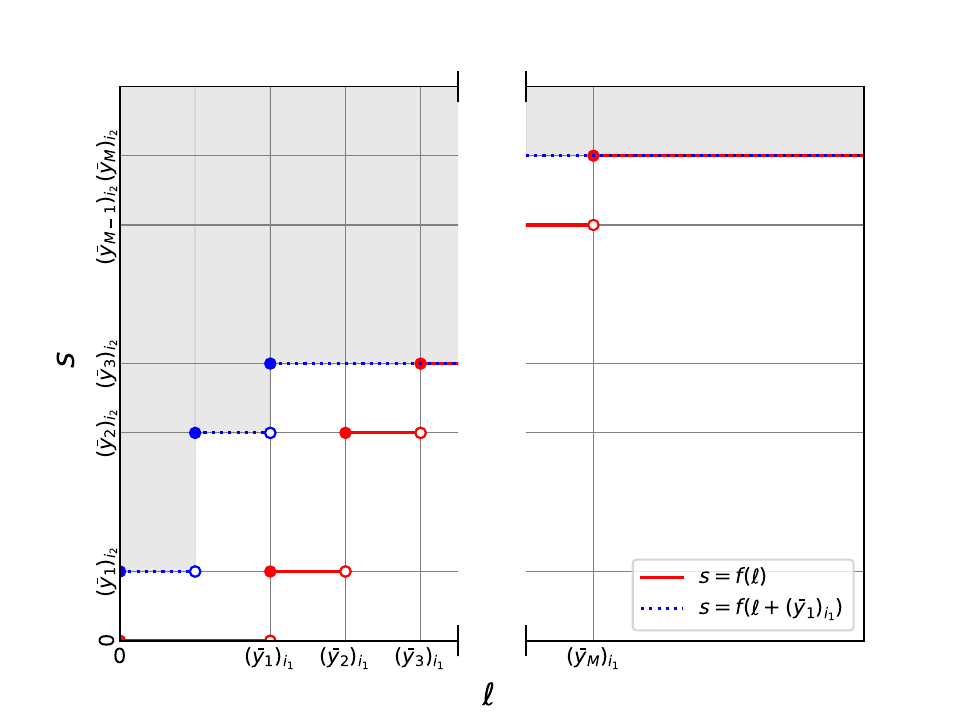}
    \caption{The red sold lines and the blue dotted lines are the graphs of $s=f(\ell)$ and $s=f(\ell+(\bar y)_{i_1})$.  The complexes $y\in \C\setminus\{\bar y_1, \dots, \bar y_M\}$ are in the shaded area.}
    \label{fig:3}
\end{figure}

 \begin{example}\label{ex: structure3}
     The reaction network in \eqref{eq:main reaction network example} satisfies the assumptions of Corollary \ref{cor:structure3} with $y^*_1=\bar y_1=\emptyset, y^*_2=\bar y_2=A+B, i_1=1$ and $i_2=2$.
 \end{example}

Now using $f$ in \eqref{eq:function f} and the concept of a \emph{subnetwork}, we can further extend the classes of reaction networks admitting non-exponential ergodicity.

\begin{cor}\label{cor:structure4}
Let $(\Sp,\C,\Re)$ and $(\bar \Sp, \bar \C, \bar \Re)$ be weakly reversible reaction networks such that $\bar \Sp=\{S_1,S_2\}$, $\Sp=\{S_1,S_2,\dots,S_d\}, \bar \C \subseteq \C$, and $\bar \Re \subseteq \Re$. Suppose that $(\bar \Sp, \bar \C, \bar \Re)$  satisfies the conditions of Corollary \ref{cor:structure3} with $\bar y^*_m$'s, 
$\bar y_m$'s, $i_1, i_2$ and function $f$ defined as in Corollary \ref{cor:structure3}. For each $y\in \C\setminus \bar \C$, we assume that either $(y)_{i_2} > f((y)_{i_1} + (\bar y_1)_{i_1})$ or $(y)_k > 0$ for some $k \in \{3,4, \dots , d\}$.  Let $x=y^*_{m'}$ for some $m'$. If under a choice of the parameters $\kappa$, the associated Markov chain $X$ for $(\Sp,\C,\Re,\kappa)$ with $X(0)=x$ is ergodic on $\mathbb S_x$, then $X$ is non-exponentially ergodic.
\end{cor}

\begin{example}\label{ex:extention by subnetworks}
     Consider the reaction network $(\Sp, \C, \Re)$.
\begin{align}
    \emptyset \rightleftharpoons A+B, \quad C \rightleftharpoons B \rightleftharpoons 2B.
\end{align}
By the same argument as in Example \ref{ex:structural1}, the associated Markov chain $X$ is ergodic under any choice of parameters. 

Let $(\bar \Sp,\bar \C, \bar \Re)$ be the reaction network in \eqref{eq:main reaction network example} that is a subnetwork of $(\Sp,\C,\Re)$ and satisfies the assumptions of Corollary \ref{cor:structure3} as shown in Example \ref{ex: structure3}.
If we let $S_1 = A$, $S_2 = B$, $S_3 = C$, $y^*_1 = \emptyset$, and $y^*_2 = A+B$, then $\bar y_1 = y^*_1$, and $\bar y_2 = y^*_2$. In this case, $f(\ell) = 1$ for $\ell \in [1,\infty)$ and $f(\ell)=0$ otherwise. Then for $ y\in \C \setminus \bar \C = \{C\}$, it holds that $(y)_3 > 0$. Therefore, $X$ with $X(0) = \vec{0}$ is non-exponentially ergodic.
\hfill $\triangle$
\end{example}

\begin{rem}\label{rem:re enumerating}
    The results in Corollaries \ref{cor:structure1}--\ref{cor:structure4} are more general than they appear.
    
    First, if $X$ is ergodic for an initial condition $\bar x$, then the condition for the initial state can be replaced by $X(0)=\bar x\in \mathbb{S}_x$ if $x=y^*_{m'}$ for some $m'$ in Corollaries \ref{cor:structure1} -- \ref{cor:structure4}. If $\mathbb S_x=\mathbb Z^d_{\ge 0}$ for any $x$, then the initial state can be arbitrary in Corollaries \ref{cor:structure1} -- \ref{cor:structure4}.

    Furthermore if a given reaction network is weakly reversible (Definition \ref{def:WR}), then there exists a choice of parameters $\kappa$ under which the associated Markov chain is ergodic. As reaction networks satifying the conditions in Corollaries \ref{cor:structure1} and \ref{cor:structure2} are necessarily weakly reversible, ergodicity of $X$ can always hold under a certain choice of parameters.
\end{rem}

There are more general structural criteria for non-exponential ergodicity, but due to its technicality, we provide it later as Theorem \ref{thm: structural condition} in Section \ref{subsec:technical}. In fact, the main ideas of the proofs of Corollary \ref{cor:structure1}--\ref{cor:structure3} are to use Theorem \ref{thm: structural condition}.

\subsection{Extension via coupling}\label{subsec:coupling}
We now extend the classes of non-exponentially ergodic stochastic reaction networks using coupling of two Markov chains.

For instance, consider two ergodic Markov chains $X$ and $\bar X$ associated with two reaction systems $(\Sp,\C,\Re,\kappa)$ and $(\bar \Sp,\bar \C,\bar \Re,\bar \kappa)$, respectively. Suppose that $\bar \Re \subseteq \Re$ (hence $\bar \Sp \subseteq \Sp$ and $\bar \C\subseteq \C$) and $\bar \kappa_{y\to y'}=\kappa_{y\to y'}$ for each $y\to y'\in \bar \Re$. Then, by the random time representation \cite{AndKurtz2011, Kurtz80}, 
\begin{align}
\begin{split}\label{eq:coupling}
    &X(t)=X(0)+\sum_{y\to y' \in \Re} Y_{y\to y'}\left ( \int_0^t \kappa_{y\to y'}\lambda_y(X(s))ds\right ) (y'-y)\\
    & \bar X(t)= \bar X(0)+\sum_{y\to y' \in \bar \Re} Y_{y\to y'}\left ( \int_0^t \kappa_{y\to y'} \lambda_y(\bar X(s))ds\right ) (y'-y)
\end{split}
\end{align}
in distribution, where $Y_{y\to y'}$'s are independent unit Poisson processes.
In \eqref{eq:coupling} we may abuse the notation in $y'-y \in \bar \Re$ as the dimensions of $X$ and $\bar X$ might be different. 
Note that as the unit Poisson processes $Y_{y\to y'}$ for $y\to y'\in \bar \Re$ are shared, $X$ and $\bar X$ are coupled. If each reactions in $\Re \setminus \bar \Re$ cannot change the copy numbers of $\bar \Sp$ and $(X(0))_i=(\bar X(0))_i$ for $S_i\in \bar S_i$, then under the coupling \eqref{eq:coupling}, the returning time to $\bar X(0)$ for $\bar X$ is shorter than or equal to the returning time to $X(0)$ for $X$ almost surely, because additional reactions in $\Re$ cannot create and remove the species $S_i$'s in $\bar \Sp \subseteq \Sp$. Hence if $\bar X$ is non-exponentially ergodic, then so $X$ is by Theorem \ref{thm: kingman}. 

This approach can be applied to general classes of Markov chains on countable state spaces. In the following proposition, we will consider two Markov chains $X$ and $\bar X$ defined on $\mathbb Z^d$ and $\mathbb Z^{\bar d}$, respectively, with $d > \bar d$. We enumerate the index of the coordinates as $\{1,2,\dots,\bar d,\bar d+1,\dots,d\}$.
 For any pairs of states $z,w \in \mathbb S_x$ and $\bar z, \bar w \in \bar{\mathbb S}_{\bar x}$, we denote by $q_{z,w}$ and $\bar q_{\bar z, \bar w}$ the transition rates of $X$ and $\bar X$, respectively. We will denote the state space of $X$ and $\bar X$ by $\mathbb S$ and $\bar{\mathbb S}$, respectively.
\begin{prop}\label{prop:coupling}
     Suppose that for any pair of states $z$ and $w$ in $\mathbb S_x$, the transition rate $q_{z,w}$ satisfies that 
\begin{align}\label{eq:transition of coupling}
        q_{z, w}= \begin{cases}
            \bar q_{\bar z, \bar w} \quad &\text{if $w_i-z_i=0$ for each $i\ge \bar d+1$,}\\
            0 &\text{if $w_i-z_i\neq 0$ for some $i\le \bar d$ and $w_j-z_j\neq 0$ for some $j\ge  \bar d+1$,}
        \end{cases}
    \end{align}
    where states $\bar z$ and $\bar w$ in $\bar{\mathbb S}_{\bar x}$ are such that $(\bar z)_i=z_i$ and $(\bar w)_i=w_i$ for each $i\le \bar d$.    For any initial conditions $X(0)=x$ and $\bar X(0)=\bar x$ such that $(x)_i=(\bar x)_i$ for $i\le \bar d$, if $\bar X$ is non-exponentially ergodic so $X$ is. 
\end{prop}
\begin{rem}
    As described for $X$ and $\bar X$ in \eqref{eq:coupling}, the condition \eqref{eq:transition of coupling} implies that $X$ and $\bar X$ have the same transition rates if the transition only alter both the first $\bar d$ coordinates of $X$, and each transition of $X$ cannot alter the first $\bar d$ coordinates and the remaining coordinates at the same time.
\end{rem}

\begin{proof}
    Define $\hat X \in \Z^{\bar d}_{\ge 0}$ such that $(\hat X(t))_i = (X(t))_i$ for each $i\le \bar d$ and $t \ge 0$. 
     For any states $x_0\in \mathbb S_x$ and $\bar x_0\in \bar{\mathbb S}_{\bar x}$ such that $(x_0)_i=(\bar x_0)_i$ for each $i\le \bar d$, let the returning time for $X$ and $\hat X$ to $x_0$ and $\bar x_0$ be $\tau_{x_0}$ and $\bar \tau_{\bar x_0}$, respectively. By definition of $\hat X$, we have $\tau_{x_0}\ge \bar \tau_{\bar x_0}$ almost surely. Therefore  if we show that $\bar X$ and $\hat X$ are equal in distribution, then non-exponential ergodicity of $\bar X$ implies non-exponential ergodicity of $X$ by Theorem \ref{thm: kingman}.  To do that we show that the distributions of $\bar X(t)$ and $\hat X(t)$ solve the same Kolmogorov forward equation at each time $t>0$.

For the fixed initial state $x$, we define $p(z,t)=P(X(t)=z)$. Similarly, for the fixed initial state $\bar x$, we define $\bar p(\bar z,t)=P(\bar X(t)=\bar z)$ and $\hat p(\bar z,t)=P(\hat X(t)=\bar z)$. 
    Since $\bar X$ is non-explosive, the probability distribution $\bar p(\bar z,t)$ is the unique solution of the following Kolmogorov forward equation \cite[Chapter 2.8]{NorrisMC97}: for each $\bar z\in \bar{\mathbb S}_{\bar x}$
    \begin{align}\label{eq:cme of bar x}
        \frac{d}{dt }\bar p(\bar z, t)= \dsum_{\bar z' \in \bar{\mathbb S}_{\bar x}\setminus \{\bar z\}} \bar p(\bar z',t)\bar q_{\bar z',\bar z}- \sum_{\bar z' \in \bar{\mathbb S}_{\bar x}\setminus \{\bar z\}} \bar p(\bar z,t)\bar q_{\bar z, \bar z'}.
    \end{align}
    Hence, it suffices to show that $\hat p(\bar z,t)$ is a solution to \eqref{eq:cme of bar x} for each $\bar z\in \bar{\mathbb S}_{\bar x}$.

    We first fix $\bar z \in \bar{\mathbb S}_{\bar x}$. Let $U(\bar z):=\{ z \in \mathbb S: (z)_i=(\bar z)_i \text{ for each $i\le \bar d$}\}$. Then as $\hat p(\bar z,t)  = \dsum_{z\in U(\bar z)}p(z,t)$.
    we have that for each $t > 0$,  
\begin{align}\label{eq:interchange}
     \frac{d}{dt }\hat p(\bar z,t) = \dsum_{z\in U(\bar z)}\dfrac{d}{dt}p(z,t),
    \end{align}
    where the interchangeability of the summation and the differentiation holds due to the uniform convergence of $\dsum_{z\in U(\bar z)}\dfrac{d}{dt}p(x,t)$ \cite[Section $\uppercase\expandafter{\romannumeral2\relax}.3$]{chung1967markov}.
    By summing up the Kolmogorov forward equation of $X$ over $z\in U(\bar z)$, we have that by \eqref{eq:interchange} 
    \begin{align}
        \frac{d}{dt }\hat p(\bar z,t)  &= \dsum_{z\in U(\bar z)} \left ( \dsum_{z' \in \mathbb S_{x}\setminus \{ z\}}  p(z',t) q_{ z', z}- \sum_{ z' \in \mathbb S_{x}\setminus \{ z\} } p( z,t) q_{z,z'}\right )       
        \notag \\ 
        &=  \dsum_{z' \in \mathbb S_{x}\setminus \{ z\} }   \dsum_{z\in U(\bar z)} p(z',t) q_{ z', z}- \dsum_{z\in U(\bar z)}  \dsum_{z' \in \mathbb S_{x}\setminus \{ z\}} p( z,t) q_{z,z'} 
        \notag \\
        &= \dsum_{z' \not \in U(\bar z) }  \dsum_{z\in U(\bar z)} p(z',t) q_{ z', z} + \dsum_{\substack{z',z \in U(\bar z) \\ z' \neq z}}  p(z',t) q_{ z', z} - \dsum_{z\in U(\bar z)}  \dsum_{z' \in \mathbb S_{x}\setminus \{ z\}} p( z,t) q_{z,z'}\label{cme}
    \end{align} 
For $z' \not \in U(\bar z)$ and $z \in U(\bar z)$, by the definition of the transition rates, $q_{z',z} > 0$ if and only if $(z')_i = (z)_i$ for each $i \ge \bar d+1$. In this case, $q_{z',z} = \bar q_{\bar z', \bar z}$, where $\bar z'\in \bar{\mathbb S}_{\bar x}$ such that $(z')_i=(\bar z')_i$ for $i\le \bar d$.
    Hence, from the first summation of \eqref{cme}, we have 
    \begin{align}\label{1st}
        \dsum_{z' \not \in U(\bar z) }   \dsum_{z\in U(\bar z)} p(z',t) q_{ z', z} = \dsum_{z' \not \in U(\bar z) }    p(z',t) \bar q_{\bar z', \bar z}=\dsum_{\bar z' \in \bar{\mathbb S}_{\bar x}\setminus \{\bar z\} }    \hat p(\bar z',t) \bar q_{\bar z', \bar z}
    \end{align}
    For the second summation of \eqref{cme}, by simply switching the indices, we have
    \begin{align}\label{2nd}
          \dsum_{\substack{z',z \in U(\bar z) \\ z' \neq z}}  p(z',t) q_{ z', z} = \dsum_{\substack{z,z' \in U(\bar z) \\ z \neq z'}}  p(z,t) q_{ z, z'} = \dsum_{z \in U(\bar z)} \dsum_{z' \in U(\bar z) \setminus \{z\}} p(z,t) q_{ z, z'}
    \end{align}
    By plugging \eqref{1st} and \eqref{2nd} into \eqref{cme}, we have 
    \begin{align*}
        \frac{d}{dt }\hat p(\bar z,t) &= \dsum_{\bar z' \in \bar{\mathbb S}_{\bar x}\setminus \{\bar z\} }    \hat p(\bar z',t) \bar q_{\bar z', \bar z} -\dsum_{z \in U(\bar z)} p(z,t) \left(  \dsum_{z' \in \mathbb S_{x}\setminus \{ z\}} q_{z,z'} - \dsum_{z' \in U(\bar z) \setminus \{z\}} q_{ z, z'} \right) 
        \\ 
        &=\dsum_{\bar z' \in \bar{\mathbb S}_{\bar x}\setminus \{\bar z\} }    \hat p(\bar z',t) \bar q_{\bar z', \bar z} -\dsum_{z \in U(\bar z)} p(z,t)   \dsum_{z' \not \in U(\bar z)} q_{z,z'} 
        \\
        &= \dsum_{\bar z' \in \bar{\mathbb S}_{\bar x}\setminus \{\bar z\} } \hat p(\bar z',t) \bar q_{\bar z', \bar z} - \dsum_{z \in U(\bar z)} p(z,t)\dsum_{\bar z' \in  \bar{\mathbb S}_{\bar x} \setminus \{ \bar z\}} 
         \bar {q}_{\bar z, \bar z'} \\ &= \dsum_{\bar z' \in \bar{\mathbb S}_{\bar x}\setminus \{\bar z\} } \hat p(\bar z',t) \bar q_{\bar z', \bar z} -  \hat p(\bar z,t)\dsum_{\bar z' \in  \bar{\mathbb S}_{\bar x} \setminus \{ \bar z\}} 
         \bar {q}_{\bar z, \bar z'},
    \end{align*}
    where the third equality holds because for $z'\neq z$, \eqref{eq:transition of coupling} implies that $q_{z,z'}=q_{\bar z, \bar z'}$ only if $(\bar z')_i = (\bar z')_i$ for $i\ge \bar d+1$  and $q_{\bar z, \bar z'}=0$ otherwise.
\end{proof}

This coupling approach further extends the classes of reaction networks that admit non-exponential ergodicity. 
To do this, we use a subnetworks whose species `catalyze' the other reactions.  
\begin{cor}\label{cor:subnetwork}
     Let $(\Sp,\C,\Re)$ and $(\bar \Sp, \bar \C, \bar \Re)$ be reaction networks such that  $\bar \Sp \subseteq \Sp, \bar \C \subseteq \C$, and $\bar \Re \subseteq \Re$. Enumerating $\bar \Sp=\{S_1,\dots,S_{\bar d}\}$ and $\Sp=\{S_1,\dots,S_{\bar d},S_{\bar d+1},\dots,S_d\}$, we suppose that
     \begin{enumerate}
         \item each species $S_i \in \bar \Sp$ satisfies that $y_i=y'_i$ for any $y\to y'\in \Re\setminus \bar \Re$ (i.e. species in $\bar \Sp$ only catalyze the reactions outside $\bar \Re$).
         \item under a choice of the parameters $\bar \kappa$ the associated Markov chain $\bar X$ for $(\bar \Sp, \bar \C, \bar \Re, \bar \kappa)$ with $\bar X(0)=\bar x$ is ergodic on $\bar{\mathbb S}_{\bar x}$.
     \end{enumerate}
Let $x$ be such that $(x)_i=(\bar x)_i$ for each $i\le \bar d$. 
If under any choice of the parameters $\kappa$ such that $\kappa_{y\to y'}=\bar \kappa_{y\to y'}$ for each $y\to y'\in \bar \Re$, the associated Markov chain $X$ for $(\Sp,\C,\Re,\kappa)$ 
 with $X(0)=x$ is ergodic on the state space $\mathbb S_x$, then $X$ is non-exponentially ergodic.  
\end{cor}
\begin{proof}
For each $y\to y'\in \bar \Re \subseteq \Re$, $X$ and $\bar X$ share the same reaction intensity. The reactions in $ \Re \setminus \bar \Re$ do not alter the species in $\bar \Sp$. Therefore, the transitions rates of $X$ and $\bar X$ satisfy \eqref{eq:transition of coupling}. Hence non-exponential ergodicity of $X$ holds by Proposition \ref{prop:coupling}.
\end{proof}
Corollary \ref{cor:subnetwork} combined with Corollaries \ref{cor:structure1}--\ref{cor:structure4} can provide broader classes of reaction networks admitting non-exponential ergodicity. More interestingly, using this extension, we can create many detailed balanced Markov models that admit non-exponential ergodicity. We will discuss this in Section \ref{sec:detail complex}.
\begin{example}\label{ex:detailed balanced 2}
Let $(\Sp,\C,\Re,\kappa)$ be a reaction system described with the reaction network below.
    \begin{align}
    \emptyset \xrightleftharpoons[\kappa_2]{\kappa_1} A+B, \quad B \xrightleftharpoons[\kappa_4]{\kappa_3} 2B, \quad A\xrightleftharpoons[\kappa_6]{\kappa_5} C+A.
\end{align}
Let $(\bar \Sp, \bar \C, \bar \Re, \bar \kappa)$ be a reaction system give by the subnetwork formed with the first two connected components. As shown in Example \ref{ex:the main example 1}, for any initial state, the associated Markov chain for $(\bar \Sp, \bar \C, \bar \Re)$ is non-exponential ergodic for any choice of parameters. Also species $A$ catalyzes the reactions in $\{A\to C+A, C+A\to A\}=\Re\setminus \bar \Re$ meaning that those reactions do not alter neither $A$ nor $B$.   Therefore for any $x\in \mathbb Z^3_{\ge 0}$, the associated Markov chain $X$ with $X(0)=x$ is non-exponentially ergodic by Corollary \ref{cor:subnetwork}.  \hfill $\triangle$
\end{example}


\subsection{Non-exponentially ergodic biochemical systems with detailed balance and complex balance}\label{sec:detail complex}

In the previous section, we provide several structural conditions of reaction networks that ensure non-exponential ergodicity of the associated Markov models when it is ergodic. In this section, we highlight that using those structural conditions, we can construct non-exponentially ergodic, detailed balanced Markov chains. Since there are two concepts of detailed balancing: detailed balanced Markov chains and detailed balanced reaction systems, we define them precisely. 
\begin{defn}
    A reaction system $(\Sp,\C,\Re,\kappa)$ under a choice of parameters $\kappa$ is \emph{detailed balanced} if  $(\Sp,\C,\Re)$ is reversible (i.e. if $y\to y'\in \Re$ then $y'\to y\in \Re)$ and there exists $c\in \mathbb R^d_{> 0}$ such that $\kappa_{y\to y'}\prod_{i=1}^d c_i^{y_i}=\kappa_{y'\to y}\prod_{i=1}^d c_i^{y'_i}$ for each $y\rightleftharpoons y' \in \Re$.  
\end{defn}
The quantity $\kappa_{y\to y'}\prod_{i=1}^d c_i^{y_i}$ is the intensity of a reaction $y \xrightarrow{\kappa_{y\to y'}} y'$ in the  deterministically modeled reaction system under mass-action kinetics \cite{feinberg2019foundations}.

The associated Markov chain for a detailed balanced reaction system $(\Sp,\C,\Re,\kappa)$ is detailed balanced (or \emph{reversible}) \cite[Theorem 4.5]{AndProdForm}, meaning that there exists a stationary distribution $\pi$ such that $q_{z,w}\pi(z)=q_{w,z}\pi(w)$ for each pair of states $z$ and $w$. We can easily check that the reaction systems in Examples \ref{ex:strong tier 1}, \ref{ex:extention by subnetworks}, and \ref{ex:detailed balanced 2} under an arbitrary parameter set $\kappa$ are detailed balanced. Especially, as the reaction networks in Examples \ref{ex:extention by subnetworks} and \ref{ex:detailed balanced 2} constructed,  using Corollary \ref{cor:subnetwork} we can add some reversible reactions to the reaction network \eqref{eq:main reaction network example}, $\emptyset \rightleftharpoons A+B, B\rightleftharpoons 2B$, in order to construct other non-exponentially ergodic, detailed balanced Markov chains. For example, consider the following three reaction networks.
\begin{align}
    & \emptyset \rightleftharpoons A+B, \quad B \rightleftharpoons2B, \quad A\rightleftharpoons C+A  \rightleftharpoons D+A. \label{eq:detail 1}\\
     &C \rightleftharpoons \emptyset \rightleftharpoons A+B, \quad B \rightleftharpoons 2B, \quad 
A\rightleftharpoons D+A.\\
&\emptyset \rightleftharpoons A+B, \quad C+B \rightleftharpoons  B \rightleftharpoons 2B \label{eq:detail 3}
\end{align}
Then, by Corollary \ref{cor:subnetwork} the Markov chains associated with the reaction networks in \eqref{eq:detail 1}--\eqref{eq:detail 3} 
under any choices of parameter set $\kappa$ are non-exponentially ergodic and detailed balanced. 
Importantly, all the above examples are bimolecular systems (each complex consists of two or fewer species), which may exist in real life or can be synthesized.

There is another special balance condition for reaction systems: complex balance \cite{Feinberg72, feinberg2019foundations, Horn72, HornJack72}.
\begin{defn}
    A reaction system $(\Sp,\C,\Re,\kappa)$ under a choice of parameters $\kappa$ is \emph{complex balanced} if  $(\Sp,\C,\Re)$ is weakly reversible and there exists $c\in \mathbb R^d_{> 0}$ such that \begin{align*}
        \dsum_{y\to y' \in \Re_y}\kappa_{y\to y'}\prod_{i=1}^d c_i^{y_i}=\dsum_{y'\to y'' \in \Re: y''=y} \kappa_{ y'\to y''}\prod_{i=1}^d c_i^{y'_i}\quad \text{for each $y\in \C$.}
    \end{align*} 
\end{defn}
In other words, complex balanced reaction systems has a positive steady state $c$ at which the in-coming and out-going reactions from and to $y$ for each complex $y$ are balanced. Interestingly, for any complex balanced reaction systems, the associated Markov chain is ergodic and the stationary distributions is a product form of Poisson distributions  \cite{anderson2018non, AndProdForm}. Detailed balancing is a special case of complex balancing.

Note that weakly reversible reaction networks can be easily formed with reactions forming cycles. Also, a well-known structural condition, so-called \emph{zero deficiency}, yields complex balancing for any weakly reversible reaction networks under arbitrary choices of parameters. Using these two observations, we can construct many non-exponentially ergodic and complex balanced Markov chains using strong tier-1 cycles. Before giving relevant examples, we give more precise description about the zero deficiency. Deficiency $\delta$ of a reaction network $(\Sp,\C,\Re)$ is an index solely determined by its structure as it is defined by $\delta=n-\ell-s$, where $n$ and $\ell$ are the numbers of complexes and the connected components of $(\Sp,\C,\Re)$, respectively, and $s$ is the dimension of a vector space $span\{y'-y: y\to y'\in \Re\} \subseteq \mathbb R^d$. Zero deficiency ensures that the reaction system is complex balanced if the underlying reaction network is weakly reversible \cite{Feinberg87, Feinberg95a, FeinbergLec79}.

The reaction networks in Examples \ref{ex:structural1} and \ref{ex: structure 2} are weakly reversible and have zero deficiency. Using Corollaries \ref{cor:structure1} -- \ref{cor:structure4} combined with Corollary \ref{cor:subnetwork}, we can create many more complex balanced reaction systems for which the associated Markov chains are non-exponentially ergodic. For example, consider the following reaction network. 
\begin{align}\label{eq:complex non expo}
\begin{split}
    &\emptyset \ \ \longrightarrow \ \ 2A+B \\[-0.5ex]
    &  \quad \displaystyle \nwarrow \hspace{1cm}  \swarrow  \quad \qquad A +C \rightleftharpoons A+D\\[-0.5ex]
& \hspace{0.5cm} 3A+2B \hspace{1.2cm} 
\end{split}
\end{align}  
The subnetwork given by the first component in \eqref{eq:complex non expo}
is complex balanced as it is weakly reversible and has zero deficiency as $n=5, \ell=2$, and
\begin{align*}
    s=\text{dim}\{\text{span}\{(2,1,0,0),(1,1,0,0),(-3,-2,0,0),(0,0,-1,1),(0,0,1,-1)\}\}=3.
\end{align*}
Also the subnetwork satisfies the conditions of Corollary \ref{cor:structure2}. The reactions in the second component of \eqref{eq:complex non expo} cannot change the number of copies of $A$ and $B$. Finally, the entire reaction network in \eqref{eq:complex non expo} is weakly reversible and has zero deficiency. Hence the associated Markov chain is non-exponentially ergodic by Corollary \ref{cor:subnetwork}.

Without relying on Corollary \ref{cor:structure1}--\ref{cor:structure4}, we can also construct a complex balanced reaction network that admits non-exponential ergodicity directly using Theorem \ref{thm: stong tier1 and irr is enough}.
\begin{example}
Consider the weakly reversible reaction network $(\Sp, \C, \Re)$ below.
    \begin{align}
\begin{split}
    &\emptyset \ \ \longrightarrow \ \ A+B \\[-0.5ex]
    &  \quad \displaystyle \nwarrow \hspace{0.7cm}  \swarrow  \qquad B \rightleftharpoons 2B \qquad D \rightleftharpoons 2D \\[-0.5ex]
& \hspace{0.7cm} C+D \hspace{1.2cm} 
\end{split}
\end{align} 
The deficiency of this reaction network is zero. Hence, by \cite[Theorem 4.2]{AndProdForm}, for any choice of parameters $\kappa$, $(\Sp, \C, \Re, \kappa)$ is complex balanced. 
We can also check that $R = (\emptyset \to A+B, A+B \to C+D, C+D \to \emptyset)$ is a strong tier-1 cycle along $\{(n,0,n^2,0)\}$. Furthermore, $\Z^4_{\ge 0}$ is irreducible, and hence for any choice of parameters $\kappa$, the associated Markov chain $X$ is non-exponentially ergodic by Theorem \ref{thm: stong tier1 and irr is enough}. 

 \hfill $\triangle$
\end{example}

\section{Proofs of the results in Sections \ref{sec:strong} and \ref{subsec:structural}}\label{sec:proofs}

In this section, we present the proofs of the theorems, the corollaries, and the proposition shown in Sections \ref{sec:strong} and \ref{subsec:structural}. The proofs of Corollaries \ref{cor:structure1}--\ref{cor:structure4} rely on Theorem \ref{thm: structural condition} that provide structural conditions for the existence of a strong tier-1 cycle.  This theorem is introduced In Section \ref{subsec:technical}.

\subsection{Proof of Proposition \ref{prop:strong tier 1}}\label{subsec:proof1}
To show result 1, we first note that the condition $\lambda_{y^*_1}(x_n)>0$ of a strong tier-1 cycle implies that $x_n\ge y^*_1$. Hence $x^2_n=x_n+y^*_2-y^*_1 \ge y^*_2$. Inductively, we can show that $x^m_n\ge y^*_m$ for each $m$. Hence $\lambda_{y^*_m}(x^m_n)>0$ for all $n$. Then for a fixed $m \in \{1,2,\dots,M\}$, let $y \in C_s \setminus \{y^*_m\}$. Then we observe that 
	\begin{align}\label{eq:strong-tier feature 1}
		\lim_{n \to \infty} \frac{\lambda_y(x_n^m)}{\lambda_{y^*_m}(x_n^m)} = \lim_{n \to \infty} \frac{\lambda_y(x_n^m)\lambda_{y^*_{m_0}}(x_n^{m_0})}{\lambda_{y^*_m}(x_n^m)\lambda_{y^*_{m_0}}(x_n^{m_0})}=  \lim_{n \to \infty} \frac{\lambda_y(x_n^m)\lambda_{y^*_{m_0}}(x_n^{m_0})}{\lambda_{y^*_m}(x_n^m)}\lim_{n \to \infty} \frac{1}{\lambda_{y^*_{m_0}}(x_n^{m_0})}=0,
	\end{align}
because $x_n^{m_0} \ge y^*_{m_0}$ so that $\lambda_{y^*_{m_0}}(x_n^{m_0}) > 0$ for all $n$. For the last equality, we used \eqref{eq:key for strong tier 1}. Hence, result 1 follows.

Next, if there exists $y^*_m \to y' \in \Re_{y^*_m}\setminus \{y^*_m\to y^*_{m+1}\}$,  then \eqref{eq:strong-tier feature 1} yields that
\begin{align*}
	1 = \lim_{n \to \infty} \frac{\lambda_{y^*_m}(x_n^m)}{\lambda_{y^*_m}(x_n^m)} = \lim_{n \to \infty} \frac{\lambda_{y^*_m}(x_n^m) \lambda_{y^*_{m_0}}(x_n^{m_0})}{\lambda_{y^*_m}(x_n^m)} \lim_{n \to \infty} \frac{1}{\lambda_{y^*_{m_0}}(x_n^{m_0})}=0,
\end{align*}
leading to a contradiction.

Now we turn to show result 3. Let $y \in \C_s \setminus \{y^*_{m_0}\}$. Substituting of $m$ by $m_0$ in \eqref{eq:key for strong tier 1}, we get 
\begin{align*}
	\lim_{n \to \infty} \frac{\lambda_y(x_n^{m_0})\lambda_{y^*_{m_0}}(x_n^{m_0})}{\lambda_{y^*_{m_0}}(x_n^{m_0})} = \lim_{n \to \infty}\lambda_y(x_n^{m_0}) = 0.
\end{align*} 
We choose a subsequence of $x_{n}$ each of whose coordinate is either strictly increasing or a constant (we denote the subsequence by $\{x_n\}$ for simplicity). Then the only possible option is $\lambda_{y}(x_n^{m_0})\equiv 0$ by the definition of mass-action \eqref{mass}.  

Now we show the last result. Suppose that $\{i: \displaystyle\sup_{n}(x_n)_i< \infty \} = \emptyset$, that is, for each $i$, $\dlim_{n\to \infty} (x_n)_i = \infty$ by considering a subseqence of $\{x_n\}$ if needed. Then obviously $\dlim_{n\to \infty} (x^{m_0}_n)_i = \infty$ for each $i$. Therefore for any $y\in \C_s$, we have that $\lambda_{y}(x^{m_0}_{n})> 0$ for large enough $n$ by the definition of mass-action \eqref{mass}, which is contradictory to result 3. \hfill $\square$

\subsection{Proof of Theorem \ref{thm: stong tier1 and irr is enough}}\label{subsec:proof2}
We will show that the conditions in Theorem \ref{thm:maybe main} hold.
We first note that for a continuous-time Markov associated with a reaction system $(\Sp,\C,\Re,\kappa)$, the sum of all transition rates, $q_z$, at a state $z$ is $\dsum_{y\to y'\in \Re} \lambda_{y\to y'}(z)$. Hence under mass-action kinetics \eqref{mass},
\begin{align}\label{eq:lower bound of the intensities}
    \inf_{z}q_z=\inf_{z\in \mathbb S}\dsum_{y\to y'\in \Re} \lambda_{y\to y'}(z)\ge \min_{y\to y'\in \Re} k_{y\to y'}>0.
\end{align}
provided that no absorbing state is in the state space. 
Note that  if \eqref{eq:kingman key equality} holds for some $\rho>0$, then so it does for any $\rho'>\rho$. Hence for the sake of the simplicity, we  replace the condition $\rho< \displaystyle \inf_z q_z$ in Theorem \ref{thm:maybe main} by $\rho < \displaystyle \min_{y\to y'\in \Re} k_{y\to y'}$. Let $\rho \in (0,\displaystyle \min_{y\to y'\in \Re} k_{y\to y'})$ be fixed.

     We consider a series of paths $\gamma_n = (x^1_n, x^2_n, \cdots, x^M_n, x^1_n)$ with $n=1,2,\dots$. Note that by result 1 in Proposition  \ref{prop:strong tier 1}, we have  $\lambda_{y^*_m\to y^*_{m+1}}(x^m_n)>0$ for all $n$ and $m$, which means that for fixed $n$, the transition from $x^m_n$ to $x^{m+1}_n$ is available for any $m$. This yields that $\gamma_n$ is an active path so that $x^m_n$ is reachable from $x_n$ for any $m$ and $n$. This also means that $\{x^m_n\} \subseteq \mathbb S$ for any $m$ as $\{x_n\}\subseteq \mathbb S$. Hence by Theorem \ref{thm:maybe main}  
     it suffices to show that  there exists $N_\rho$ such that $F(\gamma_{N_\rho},\rho) >1$.

Note that we have $\lambda_y(x^{m_0}_n)=0$ for any $y\in \C_s\setminus\{y^*_{m_0}\}$ by result 3 of Proposition \ref{prop:strong tier 1}. Hence $m_0\in A:=\{m: \lambda_{y}(x^m_n)\equiv 0 \text{ for any $y\in \C_s\setminus\{y^*_m\}$}\}$. Also for each $m\in A$ we have that 
\begin{align*}
    \frac{\lambda_{y^*_m\to y^*_{m+1}}(x^m_n)}{  \displaystyle \sum_{y\to y' \in \Re} \lambda_{y\to y'} (x^m_n)  - \rho}=\frac{\lambda_{y^*_m\to y^*_{m+1}}(x^m_n)}{  \lambda_{y^*_m\to y^*_{m+1}}(x^m_n)  - \rho}=1+ \frac{\rho}{\lambda_{y^*_m\to y^*_{m+1}}(x^m_n)-\rho}.
\end{align*}
  Then by \eqref{eq:generate function on paths2} we have that
    \begin{align}
       F(\gamma_n,\rho) &= \prod_{m=1}^M \frac{\lambda_{y^*_m\to y^*_{m+1}}(x^m_n)}{  \displaystyle \sum_{y\to y' \in \Re} \lambda_{y\to y'} (x^m_n)  - \rho} \notag \\
       &= \prod_{m \in A} \left ( 1+ \frac{\rho}{\lambda_{y^*_m\to y^*_{m+1}}(x^m_n)-\rho} \right)
 \prod_{m \in A^c} \left ( 1 - \frac{  \dsum_{y\to y' \in \Re \setminus \{ y^*_m \to y^*_{m+1}\}}\lambda_{y\to y'} (x^m_n)-\rho}{ \dsum_{y\to y' \in \Re} \lambda_{y\to y'} (x^m_n)  - \rho} \right) \notag \\ 
&\ge \left ( 1+ \frac{\rho}{\lambda_{y^*_{m_0}\to y^*_{m_0+1}}(x^{m_0}_n)-\rho} \right) \prod_{m \in A^c} \underbrace{\left ( 1 - \frac{  \dsum_{y\to y' \in \Re \setminus \{ y^*_m \to y^*_{m+1}\}}\lambda_{y\to y'} (x^m_n)-\rho}{ \dsum_{y\to y' \in \Re} \lambda_{y\to y'} (x^m_n)  - \rho} \right)}_{g(m,x_n)} \label{eq:F on general conditions}.
    \end{align}   

Fix $\varepsilon>0$ to be chosen.
We now show that there exists $C>0$ such that for any $m\in A^c$, $g(m,x_n)\ge 1-C\dfrac{\varepsilon}{\lambda_{y^*_{m_0}}(x^{m_0}_n)}$ for large enough $n$. 
First of all, note that for any $y\to y'\in \Re$ and for any state $z$, $\lambda_{y\to y'}(z)\ge \displaystyle\min_{y\to y'}\kappa_{y\to y'}$ if $\lambda_{y\to y'}(z)>0$. Thus for each $m\in A^c$, we have $\displaystyle\sum_{y\to y'\in \Re}\lambda_{y\to y'}(x^m_n)-\rho \ge \lambda_{y^*_m\to y^*_{m+1}}(x^m_n)$.
Furthermore, for each $m\in A^c$, that we choose $\bar y_m \in \C_s \setminus\{y^*_m\}$ so that there exists $C_m>0$ such that $\dsum_{y\to y' \in \Re \setminus \{ y^*_m \to y^*_{m+1}\}}\lambda_{y\to y'}(x^m_n) \le C_m\lambda_{\bar y_m}(x^m_n)$. This is possible due to property 3
of tier sequences in Definition \ref{def_tier-seqence}.   Let $C=\displaystyle\max_m \dfrac{C_m}{ \kappa_{y^*_m\to y^*_{m+1}}}$ that is independent of a choice of $\varepsilon$. Then 
\begin{align}
    g(m,x_n)&\ge  1 - \frac{  C_m\lambda_{\bar y_m} (x^m_n)}{  \lambda_{y^*_m \to y^*_{m+1}}(x^m_n) } \ge 1-C\frac{\lambda_{\bar y_m} (x^m_n)}{\lambda_{y^*_m}(x^m_n)}\notag \\
    & =  1 - C\frac{  \lambda_{y^*_{m_0}}(x^{m_0}_n) \lambda_{\bar y_m} (x^m_n)}{ \lambda_{y^*_m}(x^m_n) } \frac{1}{\lambda_{y^*_{m_0}}(x^{m_0}_n)} \ge  1 - C\frac{\varepsilon}{\lambda_{y^*_{m_0}}(x^{m_0}_n)} \quad \text{for large enough $n$,} \label{eq:lower bound of g}
\end{align}
where the last inequality follows by \eqref{eq:key for strong tier 1}.

Finally, note that for any $k\ge 1$,  $(1-x)^k\ge (1-2kx)$ for sufficiently small positive $x$. Thus  from \eqref{eq:F on general conditions} and \eqref{eq:lower bound of g} with a selection of sufficiently small  $\varepsilon$ we have that
\begin{align*}
F(\gamma_n,\rho) &\ge \left ( 1+ \dfrac{\rho}{\kappa_{y^*_{m_0}\to y^*_{m_0+1}}\lambda_{y^*_{m_0}}(x^{m_0}_n)} \right) \left (1-C\dfrac{\varepsilon}{\lambda_{y^*_{m_0}}(x^{m_0}_n)} \right )^{|A^c|}  \\
&\ge \left ( 1+ \dfrac{\rho}{\kappa_{y^*_{m_0}\to y^*_{m_0+1}}\lambda_{y^*_{m_0}}(x^{m_0}_n)} \right) \left (1-\dfrac{2|A^c|C\varepsilon}{\displaystyle \min_{y\to y'\in \Re} k_{y\to y'}}\right ) \\
&\ge 1+ \dfrac{\rho}{\lambda_{y^*_{m_0}}(x^{m_0}_n)} \left( \frac{1}{\kappa_{y^*_{m_0}\to y^*_{m_0+1}}} - \frac{2|A^c|C\varepsilon}{\rho}-\frac{2|A^c|C\varepsilon}{\displaystyle \min_{y\to y'\in \Re} k_{y\to y'}} \right )\\
&\ge 1+ \dfrac{\rho}{\lambda_{y^*_{m_0}}(x^{m_0}_n)} \left( \frac{1}{\kappa_{y^*_{m_0}\to y^*_{m_0+1}}} - \frac{2|A^c|C\varepsilon}{\rho}-\frac{2|A^c|C\varepsilon}{\displaystyle \min_{y\to y'\in \Re} k_{y\to y'}} \right )
\end{align*}
for large $n$, where we used \eqref{eq:lower bound of the intensities} for the second inequality. Finally, we choose $\varepsilon$ to be sufficiently small so that $ \dfrac{1}{\kappa_{y^*_{m_0}\to y^*_{m_0+1}}} - \dfrac{2|A^c|C\varepsilon}{\rho}-\dfrac{2|A^c|C\varepsilon}{\displaystyle \inf_{y\to y'\in \Re} k_{y\to y'}}>0$ and choose $N_\rho$ so that \eqref{eq:lower bound of g} holds for this $\varepsilon$. Hence we have $F(\gamma_{N_\rho}, \rho)>1$ that completes the proof. 
$\hfill$ $\square$

 \subsection{General structural conditions for exponential ergodicity}\label{subsec:technical}
 In this section, we provide general yet technical structural conditions for the existence of a strong tier-1 cycle. These conditions were used to derive Corollaries \ref{cor:structure1}--\ref{cor:structure4}. Throughout this section, we denote by $\langle u,v \rangle=\sum_{i=1}^d (u)_i(v)_i$ the inner product between two vectors $u, v \in \mathbb R^d$.

\begin{thm}\label{thm: structural condition}
    Let $(\Sp,\C,\Re)$ be a reaction network such that there exists a subset of reactions of the form $\{y^*_1\rightarrow y^*_2, y^*_2\rightarrow y^*_3, \cdots, y^*_{M-1}\to y^*_{M}, y^*_M \rightarrow y^*_1 \} \subseteq \Re$.
    Suppose that there exists $u \in \mathbb Z ^{d}_{\ge 0}$ 
     and $m_0\in \{1,2,\dots,M\}$ such that 
     \begin{enumerate}
         \item $I:=\{i:(u)_i \neq 0\}$ is a nonempty proper subset of $\{1,2,\dots,d\}$, and
         \item for each $m \in \{1,2, \dots, M\}$, we have $\langle u, y \rangle < \langle u, y^*_m - y^*_{m_0} \rangle$  whenever there exists $y\to y'\in \Re\setminus \{y^*_m\to y^*_{m+1}\}$ such that $(y)_i\le (y^*_m)_i$ for each $i\in I^c$.
     \end{enumerate} 
     Then $R=(y^*_1\rightarrow y^*_2, y^*_2\rightarrow y^*_3, \cdots, y^*_M \rightarrow y^*_1)$ is a strong tier-1 cycle along a tier sequence $\{x_n\}$, where for each $n$, $(x_n)_i=n^{(u)_i}+(y^*_1)_i$ for each $i \in I$ and $(x_n)_i=(y^*_1)_i$ for each $i \in I^c$.  
\end{thm}
\begin{proof}
    For $u\in \mathbb Z^d_{\ge 0}$ satisfying the above conditions, we define for each $n$, $(x_n)_i=n^{(u)_i}+(y^*_1)_i$ if $i \in I$ and $(x_n)_i=(y^*_1)_i$ otherwise. Then for all $n$, we have  $\lambda_{y^*_1}(x_n)>0$ since $(x_n)_i\ge (y^*_1)_i$ for each $i$. Hence condition 1 of Definition \ref{def:strong tier1} holds. For each $m$, let $x^m_n=x_n+\sum_{j=1}^{m-1}(y^*_{j+1}-y^*_j)$ and let $x^1_n=x_n$ for each $n$ as usual. We fix $m \in \{1,\dots,M\}$ throughout this proof.

     Now we turn to show condition 2 in Definition \ref{def:strong tier1}. First of all, by the definition of $\{x^m_n\}$, we have $x^m_n \ge y^*_m$. Hence $\lambda_{y^*_m}(x^m_n)>0$ for all $n$. 
      Suppose that there exists $y \to y' \in \Re\setminus\{y^*_m\to y^*_{m+1}\}$.  Since $(x^m_n)_i=(y^*_m)_i$ for each $i\in I^c$, if $(y)_i > (y^*_m)_i$ for some $i\in I^c$,  $\lambda_y(x^m_n)\equiv0$ by definition of mass-action \eqref{mass}. Thus  \eqref{eq:key for strong tier 1} holds.  On contrary, if $(y)_i \le (y^*_m)_i$ for all $i\in I^c$, \eqref{eq:key for strong tier 1} still holds because by the second hypothesis of Theorem \ref{thm: structural condition} we have that
       \begin{align*}
 \lim _{n\rightarrow \infty} \dfrac{\lambda_{y^*_{m_0}}(x^{m_0}_n) \lambda_{y}(x^m_n)}{\lambda_{y^*_m}(x^m_n)}&\ge 
 c\lim_{n\to \infty}\frac{\prod_{i=1}^d (x^m_n)_i^{(y)_i} (x^{m_0}_n)_i^{(y^*_{m_0})_i}}{\prod_{i=1}^d(x^m_n)_i^{(y^*_m)_i}} 
 = c'\lim_{n\to \infty}\frac{n^{\langle u, y\rangle} n^{\langle u, y^*_{m_0}\rangle}}{n^{\langle u, y^*_m\rangle}}\\
 &=c'\lim_{n\to\infty} n^{\langle u,y\rangle - \langle u,y^*_m-y^*_{m_0}\rangle}=0,
    \end{align*}
  for some $c>0$ and $c'>0$. For the first inequality above, we used the fact that $\lambda_{y^*_{m_0}}(x^{m_0}_n), \lambda_{y}(x^{m}_n)$ and $\lambda_{y^*_{m}}(x^{m}_n)$ are  non-zero polynomials with respect to $n$, respectively and their degrees are determined by $\langle u,y^*_{m_0} \rangle, \langle u, y\rangle$ and $\langle u, y^*_m \rangle$, respectively.    
  
  As of now, we showed that $R$ is a strong tier-1 cycle along $\{x_n\}$. To show that $\{x_n\}$ is a tier sequence, we note that
   each coordinate of $x^m_n$ is a polynomial in $n$. Therefore the limits in Definition \ref{def_tier-seqence} are well-defined in $[0,\infty]$. By the definition of $\{x^m_n\}$, it is simple to check that all other conditions in Definition \ref{def_tier-seqence} hold with $\{x_n\}$.  
\end{proof}

We demonstrate the way of identifying $u, R$ and $m_0$ with an example.
\begin{example}\label{ex: structural4}
Let $(\Sp, \C, \Re)$ be the following reaction network.  
    \begin{align}
\begin{split}\label{eq: structural4}
    &\emptyset \ \ \longrightarrow \ \ A+2B+C \\[-0.5ex]
    &  \quad \displaystyle \nwarrow \hspace{1.5cm}  \swarrow  \qquad \qquad \quad 2C \rightleftharpoons 3C.\\[-0.5ex]
& \hspace{0.7cm} 2A+B+2C \hspace{1.2cm} 
\end{split}
\end{align} 
We enumerate the species as $S_1 = A$,  $S_2 = B$, and  $S_3 = C$.
We notice that the reactions in the first connected component form a cycle in the reaction network. So we can first try to set $R = (y^*_1\to y^*_2, y^*_2\to y^*_3, y^*_3\to y^*_1)$ with $y^*_1=\emptyset, y^*_2=A+2B+C$, and $y^*_3=2A+B+2C$.  
To determine $m_0$, we use condition 2 of Theorem \ref{thm: structural condition}. As $y^*_1 < y^*_i$ for $i=2,3$, we must choose $m_0=1$ to hold the inequality $\langle u, y \rangle < \langle u, y^*_m-y^*_{m_0} \rangle$ for some $u\in \mathbb Z^d_{\ge 0}$. 

Then, we choose a proper subset $I$ of $\{1,2,3\}$ for which we construct $u \in \mathbb Z^3_{\ge 0}$ as $(u)_i=0$ for each $i\in I^c$. In order to have condition 2 in Theorem \ref{thm: structural condition}, it is reasonable to choose $I^c$ that minimizes the number of complexes $y$ satisfying $(y)_i\le (y^*_m)_i$ for each $i\in I^c$. Hence we set either $I^c=\{1,2\}, \{2,3\}$ or $\{1,3\}$.  Let $I^c=\{2,3\}$. Then we can set $u=(1,0,0)$ without loss of generality. For $m=1$, there is no $y$ such that $(y)_i\le (y^*_1)_i$ for $i\in I^c$. We also have that $\langle u, y \rangle < \langle u, y^*_2-y^*_1 \rangle$ for $y=\emptyset$, which is sufficient to satisfy condition 2 of Theorem \ref{thm: structural condition} for $m=2$. For $m=3$, $\langle u, y \rangle < \langle u, y^*_2-y^*_1 \rangle$ holds for all $y\in \{\emptyset, 2C\}$, which is also sufficient for condition 2 of Theorem \ref{thm: structural condition}.
\end{example}

\begin{rem}
    Regarding Theorem \ref{thm: stong tier1 and irr is enough}, to finally show non-exponential ergodicity via Theorem \ref{thm: structural condition}, the constructed sequence $\{x_n\}$ needs to be in $\mathbb S$. The simplest case for this is that 
  $\mathbb S= \mathbb Z^d_{\ge 0}$. Otherwise  the constructed sequence $\{x_n\}$ may not belong to the state space $\mathbb S$ (in fact, $\{x_n\}=\{(n,0,0)\}$ constructed in Example \ref{ex: structural4} does not lie in the state space of the reaction network in \eqref{eq: structural4}). However, as explained right after Theorem \ref{thm: stong tier1 and irr is enough} we can fix this problem under certain accessibility conditions (See Appendix \ref{appdx:prove lemma} for more details).
\end{rem}

Despite identifying such $R, m_0$ and $u$ in Theorem \ref{thm: structural condition} is not straightforward, if the given reaction network has a simple structure, its existence is guaranteed, as shown in Corollary \ref{cor:structure1}--\ref{cor:structure3}.

\subsection{Proofs of Corollaries \ref{cor:structure1}--\ref{cor:structure4}}
\label{subsec:proof3}
In this section, we prove Corollaries \ref{cor:structure1}--\ref{cor:structure4} that provided structural conditions for non-exponential ergodicity of reaction networks. The proofs basically have the same outline. 

First, we show that there exists a clear choice of $R, m_0$ and $u$ satisfying the assumptions of Theorem \ref{thm: structural condition}. Then using Theorem \ref{thm: structural condition} we show that $R$ is a strong tier-1 cycle along a tier sequence $\{x_n\}$ constructed as in the proof of Theorem \ref{thm: structural condition}. 

Second, since it is possible that $\{x_n\} \not \subseteq \mathbb S_x$, we use Lemma \ref{lem:new desired tier sequence} to make a new tier sequence $\{\bar x_n\} \subseteq \mathbb S_x$ along which $R$ is still a strong tier-1 cycle. Then non-exponential ergodicity follows by Theorem \ref{thm: stong tier1 and irr is enough}.  Note that the sequences $\{x_n\}$ to be constructed in the following proofs are tier-sequences by their definitions. Hence we can use Lemma \ref{lem:new desired tier sequence}.

For more simplicity, we can assume that the initial state is $y^*_1$ for all the colloraries without loss of generality because $y^*_{m'}$ is reachable from $y^*_1$ (i.e. $y^*_{m'} \in \mathbb S_{y^*_1}$) for any $m'$ by Lemma \ref{lem:cycle makes active paths}. As we will see, therefore, since the tier sequences $\{x_n\}$ used in the proofs of Corollaries \ref{cor:structure1}--\ref{cor:structure4} satisfies $(x_n)_i\equiv (y^*_1)_i$ for each $i\in I^c_{\{x_n\}}$, the initial state $y^*_1$ satisfies the requirement of the initial state of Lemma \ref{lem:new desired tier sequence}. Thus it suffices to check condition 1 and 2 of Lemma \ref{lem:new desired tier sequence}.
 
\begin{proof}[\textbf{Proof of Corollary \ref{cor:structure1}}]

    First, we can choose $u$ as $(u)_{i_1}=1$ for a single element $i_1 \in U$ and $(u)_i=0$ for the remaining components. Let $R = (y^*_1\rightarrow y^*_2, y^*_2\rightarrow y^*_3, \cdots, y^*_{M-1}\to y^*_{M}, y^*_M \rightarrow y^*_1)$. 
    Fix $m \in \{1,2,\dots, M\}$. Let $y \to y' \in \Re \setminus \{y^*_m \to y^*_{m+1}\}$ such that $(y)_i \le (y^*_m)_i$ for each $i \neq i_1$. By assumption 1 and 2, $(y)_i < (y^*_m)_i$ for all $i$. Then we have 
    \begin{align*}
        \langle u,y \rangle < \langle u,y^*_m \rangle = \langle u,y^*_m - y^*_{m_0} \rangle.
    \end{align*} 
    By Theorem \ref{thm: structural condition}, $R$ is a strong tier-1 cycle along $\{x_n\}$ such that $(x_n)_{i_1}=n^{(u)_{i_1}} + (y^*_1)_{i_1}$  and $(x_n)_i=(y^*_1)_i$ for each $i\neq {i_1}$. 

Now, we check the conditions of Lemma \ref{lem:new desired tier sequence}.
  By the definition of $u$ and $\{x_n\}$, we have $x_n-x_0=n^{(u)_{i_1}}e_{i_1}$ and $e_{i_1}\in span\{y'-y : y\to y'\in \Re\}$. Hence condition 1 of Lemma \ref{lem:new desired tier sequence} holds. Condition 2 of Lemma \ref{lem:new desired tier sequence} immediately follows from the total ordering given in assumption 1.
\end{proof}

\begin{proof}[\textbf{Proof of Corollary \ref{cor:structure2}}]
	
 First, without loss of generality, assume $i_1=1$ and $i_2=2$. Let $u=(1,0)$ and let $R=(y^*_1\rightarrow y^*_2, y^*_2\rightarrow y^*_3, \cdots, y^*_{M-1}\to y^*_{M}, y^*_M \rightarrow y^*_1)$. Then fix $m \in \{1,2,\dots,M\}$ and let $y \to y' \in \Re \setminus \{y^*_m \to y^*_{m+1}\}$ such that $(y)_2 \le (y^*_m)_2$. By condition 1, $y<y^*_m$, which implies that $y=\bar y_{k'}$ and $y^*_m=\bar y_{k}$ with $k>k'$. Hence, by condition 2 we have $\langle u,y \rangle = (y)_1 = (\bar y_k)_1 < (\bar y_{k'})_1 - (\bar y_1)_1 = \langle u,y^*_m- \bar y_1 \rangle$. Then letting $m_0$ such that $y^*_{m_0} = \bar y_1$, Theorem \ref{thm: structural condition} yields that $R$ is a strong tier-1 cycle along $\{x_n\}$, where $x_n = (n+(y^*_1)_1,(y^*_1)_2)$. 

    Clearly, $(\Sp, \C, \Re)$ is weakly reversible, and assuption 3 implies that $\{x_n-x_0\} \in \text{span}\{y'-y : y\to y'\in \Re\}$. Condition 2 of Lemma \ref{lem:new desired tier sequence} immediately follows from condition 1.
\end{proof}

\begin{proof}[\textbf{Proof of Corollary \ref{cor:structure3}}]
    First, without loss of generality, assume that condition 2 holds for $i_1 = 1$ and $i_2 = 2$. Let $u = (1,0)$ and let $R = (y^*_1 \to y^*_2, y^*_2 \to y^*_3, \dots, y^*_{M-1} \to y^*_M, y^*_M \to y^*_1)$. For fixed $m \in \{1,2,\dots, M\}$, let $y \to y' \in \Re\setminus\{y^*_m \to y^*_{m+1}\}$ such that $(y)_2 \le (y^*_m)_2$.

    We first assume that $y = y^*_{m'}=\bar y_{k'}$ for some $m', k' \in \{1,2,\dots,M\}$. Then by conditions 1 and 4,  $y^*_m = \bar y_{k}$ with $k > k'$. Then, by condition 2, $\langle u,y \rangle = (y)_1 = (\bar y_k)_1 < (\bar y_{k'})_1 - (\bar y_1)_1 = \langle u,y^*_m- \bar y_1 \rangle$.

    Now we assume that $y\in \C\setminus\{\bar y_1,\dots,\bar y_M\}$.  Then condition 5 implies that $f((y)_1+(\bar y_1)_1)< (y)_2$, which implies that $f((y)_1+(\bar y_1)_1)< (y^*_m)_2 $. By the definition of $f$, we must have $(y)_1+(\bar y_1)_1 <  (y^*_m)_1$, which further implies that $\langle u,y \rangle = (y)_1  < (y^*_m)_1 - (\bar y_1)_1 = \langle u,y^*_m- \bar y_1 \rangle$. Thus, with $y^*_{m_0} = \bar y_1$, Theorem \ref{thm: structural condition} yields that $R$ is a strong tier-1 cycle along $\{x_n\}$, where $x_n = (n+(y^*_1)_1, (y^*_1)_2)$.

    Now, we check the conditions of Lemma \ref{lem:new desired tier sequence}. We first notice that condition 3 implies that $\{x_n-x_0\} \in \text{span}\{y'-y : y\to y'\in \Re\}$. Condition 2 of Lemma \ref{lem:new desired tier sequence} immediately follows from condition 1.
\end{proof}

In the following proof, each $y \in \bar \C$ can be regarded as either a vector in $\mathbb Z^2_{\ge 0}$ or a vector in $\mathbb Z^d_{\ge 0}$ such that $(y)_i=3$ for $i\ge 3$. We also denote the state space of the associated Markov chains for $(\Sp,\C,\Re)$ and $(\bar \Sp, \bar \C, \bar \Re)$ by $\mathbb S \subseteq \mathbb Z^d_{\ge 0}$ and $\bar{\mathbb S} \subseteq \mathbb Z^2_{\ge 0}$, respectively.
\begin{proof}[\textbf{Proof of Corollary \ref{cor:structure4}}]
First, without loss of generality, assume that condition 2 of Corollary $\ref{cor:structure3}$ holds for $i_1 = 1$ and $i_2 = 2$. Let $u = (1,0,\dots,0)$ and let $R = (y^*_1 \to y^*_2, y^*_2 \to y^*_3, \dots, y^*_{M-1} \to y^*_M, y^*_M \to y^*_1)$. For fixed $m \in \{1,2,\dots, M\}$, let $y \to y' \in \Re \setminus \{y^*_m \to y^*_{m+1}\}$ such that $(y)_i \le (y^*_m)_i$ for each $i \in \{2,3,\dots, d\}$. That is, $(y)_2 \le (y^*_m)_2$ and $(y)_i = 0$ for each $i \ge 3$. 
Hence with the same way as the proof of Corollary \ref{cor:structure3}, we can show that $\langle u, y \rangle < \langle u, y^*_m - \bar y_1 \rangle$. Hence by Theorem \ref{thm: structural condition}, $R$ is a strong tier-1 cycle along a sequence $\{x_n\}$ such that $x_n = (n+(y^*_1)_1, (y^*_1)_2,0, \dots,0)$.

Instead of directly obtaining a new sequence $\{\bar x_n\} \subseteq \mathbb S_x$ from $\{x_n\}$ using Lemma \ref{lem:new desired tier sequence}, we first find a tier-sequence $\{\bar z_n\} \subseteq \Z^2_{\ge 0}$ of $(\bar \Sp,\bar \C, \bar \Re)$ and then construct a desired tier sequence using $\{\bar z_n\}$. 
We first note that $R$ is a strong tier-1 cycle of $(\bar \Sp,\bar \C, \bar \Re)$ along $\{z_n\} \subseteq \Z^2_{\ge 0}$, where $z_n = (n+(y^*_1)_1,(y^*_1)_2)$. 
Let $z = y^*_{1} \in \Z^2_{\ge 0}$.
Then as the proof of Corollary $\ref{cor:structure3}$, Lemma \ref{lem:new desired tier sequence} implies that there exists a tier sequence $\{\bar z_n\} \subseteq \bar{\mathbb S}_{z}$ such that $(\bar z_n)_2 \equiv (y^*_1)_2$ and $\sup_k\lVert \bar z _k - z_{n_k}\rVert_{\infty}<\infty$ for some subsequence $\{z_{n_k}\}$ of $\{z_n\}$.

Now we define $\{\bar x_n\} \subseteq \Z^d_{\ge 0}$ such that $(\bar x_n)_i \equiv (\bar z_n)_i$ for $i \in \{1,2\}$ and $(\bar x_n)_i \equiv 0$ for $i\ge 3$. Note that $R$ is a strong tier-1 cycle along $\{x_{n_k}\}$. Since $\bar z_n$ satisfies all the condition of Lemma \ref{lem: perturbation} with $\{z_{n_k}\}$, so $\{\bar x_n\}$ does with $\{x_{n_k}\}$ by the definition of $\{\bar x_n\}$.
Therefore, 
there exists a subsequence of $\{\bar x_n\}$ that is a tier sequence such that $R$ is a strong tier-1 cycle along $\{\bar x_n\}$. Finally we note that $\bar z_n$ is reachable from $z$ as $\{\bar z_n\} \subseteq \bar{\mathbb S}_{z}$. Then using the same reactions in $\bar \Re$ for the transitions from $z$ to $\bar z_n$ for each $n$, $X$ can transition from $x$ to $\bar x_n$ for each $n$ as well. Hence $\{\bar x_n\} \subseteq \mathbb S_x$ 
\end{proof}

\section{Spectral gap and congestion ratios}\label{sec:spectral}
In this section, we briefly discuss some connections of our results to the spectral gap and the congestion ratio of continuous-time Markov chains. The spectral gap of a continuous-time Markov chain $X$ on a discrete state space $\mathbb S$ is the smallest non-zero eigenvalue of $-\mathcal A$ (equivalently the operator norm of the Dirichlet form defined as $\mathcal E(f,f)=-\sum_{x\in \mathbb S} (\mathcal Af)(x) f(x)$), where $\mathcal A$ is the infinitesimal generator of $X$. The spectral gap is used to  study exponential ergodicity as a strictly positive spectral gap implies that the associated ergodic Markov chain is exponentially ergodic. 
See \cite{bakry2008rate, MR1106707, YuvalLevinMixing, Liggett89, Han16} for more details about the spectral gap methods. 

Hence we remark that the sufficient conditions for non-exponential ergodicity in this paper also imply `zero spectral gap'. We highlight again that the simple 2-dimensional detailed balanced Markov chain in Example \ref{ex:the main example 1} is non-exponentially ergodic so its spectral gap must be zero.  
Furthermore, as we showed in Section \ref{sec:detail complex} we can construct many detailed balanced, non-exponentially ergodic 
 Markov chains, which have zero spectral gaps.

In the literature of  Markov chains, the concept of so-called the congestion ratio is commonly employed to show a positive spectral gap  (see for instance \cite[Chapter 3]{saloff1997lectures} or \cite[Section 3.3]{berestycki2009eight}). For an ergodic Markov chain $X$ defined on a discrete state space $\mathbb S$ with the stationary distribution $\pi$, we first define a collection of paths,
\begin{align*}
    \Gamma=\{\gamma_{z,w} \ \text{or} \ \gamma_{w,z} \ \text{but not both}:  z,w \in \mathbb S \ \text{with $z\neq w$}\},
\end{align*}  
where $\gamma_{z,w}=(z_1,\dots,z_{|\gamma_{z,w}|}) \subseteq \mathbb S$ is a path such that $z_1=z$ and $z_{|\gamma_{z,w}|}=w$ and $\gamma_{w,z}$ is defined similarly.
Then the congestion ratio is defined as for a given $\Gamma$
\begin{align}\label{eq: congestion ratio}
        C_{cr} = \sup_{\{(s,u)\in \mathbb S \times \mathbb S : q_{s,u}>0\}} \left\{ \frac{1}{q_{s,u}\pi(s)} \sum_{z,w:(s,u)\in \gamma \in \Gamma}|\gamma|\pi(z)\pi(w) \right\},
    \end{align}
    where $\{z,w:(s,u)\in \gamma \in \Gamma \}$ means the collection of pairs of two states $z$ and $w$ for which there exists either a path $\gamma_{z,w}=(z(1),\dots,z(|\gamma_{xy}|)) \in \Gamma$ such that  $(z(i),z(i+1))=(s,u)$ or a path $\gamma_{w,z} =(w(1),\dots,w(|\gamma_{w,z}|))\in \Gamma$ such that $(w(i),w(i+1))=(s,u)$ for some $i$.  If $C_{cr}<\infty$ for some $\Gamma$, then $X$ is exponentially ergodic with $\rho=C_{cr}^{-1}$ in \eqref{eq:expo ergo} \cite[Section 3.3]{berestycki2009eight}. Hence our sufficient conditions for non-exponential ergodicity also imply that $C_{cr}=\infty$ for any choice of $\Gamma$. In \cite[Section 2.4]{anderson2023new}, it was shown that the congestion ratio of the associated Markov chain defined in Example \ref{ex:the main example 1} is infinity for any $\Gamma$. Notably, it was shown that for any $\Gamma$ either
    \begin{align*}
        &\limsup_{n\to \infty} \left (\frac{1}{q_{(n,0),(n+1,1)}\pi(n,0)} \sum_{z,w:(s,u)\in \gamma\in \Gamma}|\gamma|\pi(z)\pi(w) \right )=\infty \quad \text{or }\\ 
        &\limsup_{n\to \infty} \left (\frac{1}{q_{(n+1,1),(n,0)}\pi(n+1,1)} \sum_{z,w:(s,u)\in \gamma\in \Gamma}|\gamma|\pi(z)\pi(w) \right )=\infty,
    \end{align*}
    This means that the infinite congestion takes place on the active paths $\gamma_n=((n,0),(n+1,1),(n,0))$ such that $F(\gamma_n,\rho)>1$ as we showed in Example \ref{ex:strong tier 1}. This example, thus, shows that the conditions of Theorem \ref{thm:maybe main} are reasonable in the sense that they cause infinite congestion on the state space.

\section{Discussion}\label{sec:discussion}
In this paper, we provided a path method for showing non-exponential ergodicity of continuous-time Markov chains on a countable state space. We showed that on the series of paths satisfying the conditions of Theorem \ref{thm:maybe main}, the Markov chain spends a long time on the paths, and in turn non-exponential ergodicity follows. By using this path method, we identified several classes of stochastic reaction networks that admit non-exponential ergodicity. Remarkably, the classes are characterized with only structural features independently from the reaction rate parameters. The structural conditions in Corollaries \ref{cor:structure1}--\ref{cor:structure4} combined with the coupling method shown in Corollary \ref{cor:subnetwork} were used to create detailed balanced Markov chains that are non-exponentially ergodic. Similarly, non-exponentially ergodic, complex balanced stochastic reaction systems can be made.

One possible generalizations of the main results of this paper is to consider reaction networks under general kinetics beyond mass-action kinetics.
By modifying the exponents of ${x_n}$ in the proof of Theorem \ref{thm: structural condition}, the same results of Theorem \ref{thm: structural condition} and   Corollaries \ref{cor:structure1}--\ref{cor:structure4} still hold for some class of general kinetics including generalized mass action kinetics \cite{muller2012generalized}. Another possible generalization to consider reactions that do not form a cycle in the reaction network but can generate paths $\gamma_n$ for which $F(\gamma_n,\rho)>1$ for small enough $\rho$. For example, in the reaction network
\begin{align*}
    2A+2B\to \emptyset\to A+B, \quad 2B\rightleftharpoons 3B,
\end{align*}
 the ordered list of reactions $(\emptyset\to A+B, \emptyset\to A+B,2A+2B\to \emptyset)$ generates paths $\gamma_n=((n,0),(n+1,1),(n+2,2),(n,0))$ that hold $F(\gamma_n,\rho)>1$ for small enough $\rho$. 
  In general, however, an arbitrary ordered list of reactions does not necessarily generate a path on the state space (e.g. the ordered list $(2A+2B\to \emptyset, 2A+2B\to \emptyset)$ cannot generate a path if the path is supposed to start at $(n,2)$), while an ordered list of reactions forming a cycle in the reaction network can always form a path on the state space as we discussed in Remark \ref{rem:why cycle}.
Therefore, this generalization is not straightforward.

Our work can serve as a foundation for various future projects. First, in this paper we showed only the cases where non-exponential ergodicity is determined by structural conditions. On the contrary, for a given reaction network, we can study the cases where non-exponential ergodicity can be induced by a choice of parameters. This leads to a phase transition of biochemical systems in qualitative convergence rates (exponential convergence to non-exponential convergence and vice versa) by  certain parameters such as temperature, pressure, or other external disturbances. Also, as ergodicity is an apparent condition for non-exponential ergodicity, we can study the conditions for ergodicity of the classes of reaction networks in Corollary \ref{cor:structure1}--\ref{cor:structure4}.

\appendix
   \section{Proof of Theorem \ref{thm: kingman}}\label{app:prove kingman}
In this section, we prove Theorem \ref{thm: kingman}. The proof is obtained by extracting the key parts of the original proof given in \cite{kingman1964}. 

Throughout this section, let $X$ be an ergodic continuous-time Markov chain on a countable state space $\mathbb S$ and $X(0)=x \in \mathbb S$ almost surely. Let $\pi$ be the unique stationary distribution of $X$. Also, for the fixed initial state $x$, we denote
$p(t)=P_x(X(t)=x)$, $q=q_x$ for simplicity. Our primary tool for the proofs is the function $r(\theta):= \int^{\infty}_0 p(t)e^{-\theta t}dt$. This function is the Laplace transform of $p(t)$ or a type of \emph{resolvent} of $X$. $r(\theta)$ provides insights into the expected hitting times and long-term behaviors of the Markov chain \cite{MT-LyaFosterII, MT-LyaFosterIII}.
For each $m \ge 1$, define $T_{2m-1} := \inf \{s >T_{2m-2}: X(s) \neq x\}$ and $T_{2m} := \inf\{s >T_{2m-1}: X(s)=x\}$, where $T_0:=0$.
    Also, define $U_n := T_n - T_{n-1}$. This represents the time for $X$ to leave state $x$ for each odd $n$ and the returning time to $x$ for each even $n$.
\begin{lem}\label{lem: laplace} 
    For $\theta \in \mathbb C$ such that \normalfont{Re}$(\theta) >0$,
    \begin{align}
       \mathbb E(e^{-\theta U_2})= \int^{\infty}_0 e^{-\theta t}dF(t) = 1 + \frac{\theta}{q} + 
        \frac{1}{q  r(\theta)}
    \end{align}
    Here,  $F(t)$ is the cumulative distribution function of $U_2$.  

\begin{proof}
    $T_n$'s are finite stopping times as $X$ is recurrent. Also we have that $T_{n+1}>T_{n}$ almost surely for each $n$. By strong Markov property, $U_n$'s are independent and for each $m \ge 1$, $U_{2m-1}=U_1$ in distribution, which is exponentially distributed with the rate $q$, and $U_{2m}=U_2$ in distribution.  
Noting that $p(t)=\sum^{\infty}_{m=0} \mathbb P(T_{2m}\le t <T_{2m+1})$, we have for $\theta \in \mathbb C$ such that Re$(\theta)>0$, 
    \begin{align}\label{eq: 1}
        r(\theta) &= \int^{\infty}_0 p(t)e^{-\theta t}dt = \sum^{\infty}_{m=0}
        \int^{\infty}_0 \mathbb P(T_{2m}\le t <T_{2m+1}) e^{-\theta t} dt.
    \end{align}
    By Fubini's theorem, we have 
\begin{align}\label{eq: 2}
        \int^{\infty}_0 \mathbb P(T_{2m}\le t <T_{2m+1}) e^{-\theta t} dt =  \mathbb E \left( \int ^{T_{2m+1}}_{T_{2m}} e^{-\theta t}dt\right) = \theta^{-1} \mathbb E
        (e^{-\theta T_{2m}} - e^{-\theta T_{2m+1}})
    \end{align}
   Moreover, since $T_n=\sum_{k=1}^n U_k$ and $U_k$'s are independent, we have 
    \begin{align}
    \mathbb E (e^{-\theta T_{2m}})-\mathbb E (e^{-\theta T_{2m+1}}) &= \mathbb E (e^{-\theta \sum_{k=1}^{2m}X_k}) - \mathbb E (e^{-\theta \sum_{k=1}^{2m+1}U_k}) \notag \\&= \prod_{k=1}^{2m} \mathbb E (e^{-\theta U_k}) - \prod_{k=1}^{2m+1} \mathbb E (e^{-\theta U_k})\notag \\
    &= \left (\prod_{k=1}^{m}\mathbb E (e^{-\theta U_{2k-1}})\mathbb E (e^{-\theta U_{2k}}) \right ) \mathbb E \left ( 1- \mathbb E (e^{-\theta U_{2m+1}}) \right)\notag \\
    &= \left( \frac{q}{q+\theta} \int^{\infty}_0 e^{-\theta t} dF(t)\right)^m \left(1-\frac{q}{q+\theta}\right), \label{eq: 3}
    \end{align}
    where the last equality holds since $\mathbb E (e^{-\theta U_1}) = \frac{q}{q+\theta}$ and $\mathbb E (e^{-\theta U_2}) = \int^{\infty}_0 e^{-\theta t} dF(t)$. 
    
    By putting \eqref{eq: 3} into \eqref{eq: 2} and further putting \eqref{eq: 2} into \eqref{eq: 1}, we get
    \begin{align*}
        r(\theta) = \frac{1}{q+\theta} \sum^{\infty}_{m=0} \left( \frac{q}{q+\theta} \int^{\infty}_0 e^{-\theta t} dF(t)\right)^m
        &= \frac{1}{q+\theta}\left(1-\frac{q}{q+\theta} \int ^{\infty}_0 e^{-\theta t }dF(t)\right)^{-1}\\ 
        &= \left(q+\theta - q \int ^{\infty}_0 e^{-\theta t }dF(t)\right)^{-1},
    \end{align*}
    where the third equality holds because the absolute value of $\frac{q}{q+\theta} \int^{\infty}_0 e^{-\theta t} dF(t)$ is smaller than $1$.
    Then \eqref{eq: 1} follows for Re$(\theta)>0$.
\end{proof}

\end{lem}

\begin{lem}\label{lem: laplace2}
    Let $V$ be a positive random variable.  
    Then there exists $\infty \le \sigma_c \le 0$ such that $\mathbb E (e^{-\theta V})$ converges for \normalfont{Re}$(\theta) > \sigma_c$ and diverges for \normalfont{Re}$(\theta) < \sigma_c$. Furthermore, $\mathbb E (e^{-\theta V})$ is analytic on Re$(\theta) > \sigma_c$ and singular at $\theta = \sigma_c$, that is, $\mathbb E (e^{-\theta V})$ does not have the analytic continuation whose domain contains the point $\sigma_c$. 
\end{lem}
\begin{proof}
    See  \cite[Section $\uppercase\expandafter{\romannumeral2\relax}.5$]{widder2015laplace}.
\end{proof}
 $\sigma_c$ in Lemma \ref{lem: laplace2} is called the \emph{axis of convergence}. Now we turn to provide the proof of Theorem \ref{thm: kingman}.
\begin{proof}[\textbf{Proof of Theorem \ref{thm: kingman}}]
We will prove the result using a contradiction. Suppose that there exist constants $B(x)>0$ and $\eta>0$ such that 
\begin{align}\label{eq:assume ex ergo}
    |p(t)-\pi(x)|\le B(x) e^{-\eta t}, \text{ for each $t\ge 0$}. 
\end{align}
 Since $\mathbb E _x (e^{\rho \tau _ x}) = \infty$ for all $\rho > 0$, $\mathbb E _x (e^{-\theta \tau_x})$ has the axis of convergence $\sigma_c = 0$ and is singular at $\theta = 0$ by Lemma \ref{lem: laplace2}. 
Note that $\tau_x = U_1 + U_2$, and $U_1$ and $U_2$ are independent. Thus we have that
    $$\mathbb E(e^{-\theta \tau_x}) =
    \mathbb E(e^{-\theta U_1}) \mathbb E(e^{-\theta U_2}) = \frac{q}{\theta + q} \mathbb E(e^{-\theta U_2}) = \frac{q}{\theta + q} \int^{\infty}_0 e^{-\theta t}dF(t).$$
 By this, it is clear that $\int^{\infty}_0 e^{-\theta t}dF(t)$ has the same axis of convergence $\sigma_c = 0$ and is singular at $\theta = 0$. 
      To derive a contradiction, we will show that $\int^{\infty}_0 e^{-\theta t}dF(t)$ has an analytic continuation with a domain containing $\sigma_c=0$.

    By assumption \eqref{eq:assume ex ergo}, $p(t) = \pi(x) + g(t)$ such that $e^{\alpha t}g(t)\to 0$, as $t\to \infty$ for some $\alpha > 0$. Hence for Re$(\theta) >0$, we have
    \begin{align*}
    r(\theta) = \int^\infty_0 p(t) e^{-\theta t}dt = \int^\infty_0 \pi(x) e^{-\theta t} dt + \int^\infty_0 g(t)e^{-\theta t} dt = \frac{\pi(x)}{\theta} + \int^\infty_0 g(t) e^{-\theta t} dt.   
    \end{align*}
    Furthermore, $\int^\infty_0 g(t)e^{-\theta t} dt$ exists and analytic on $\text{Re}(\theta) > -\alpha$. (See \cite[Proposition 12]{kingman1964}).
    Thus $r(\theta)$ has the analytic continuation on the domain $\{\theta \in \mathbb C : \text{Re}(\theta) > -\alpha, \theta \neq 0\}$ with a simple pole at $\theta = 0$. 
    Define $\tilde r(\theta) := \theta r(\theta)$. Then we have  
    \begin{align}
        \int^{\infty}_0 e^{-\theta t}dF(t) = 1 + \frac{\theta}{q} -\frac{\theta}{q \cdot \tilde r(\theta)} \quad \text{for Re$(\theta)>0$}. 
    \end{align} 
     Since $\tilde r(\theta)$ is nonzero at $\theta = 0$ and analytic on Re$(\theta) > -\alpha$, we have that $\tilde r(\theta) \neq 0$ on $B_{\epsilon}(0):=\{\theta \in \mathbb C: |\theta |\le \epsilon\}$ for some $\epsilon>0$. Thus $1 + \frac{\theta}{q} - \frac{\theta}{q \cdot \tilde r(\theta)}$ is well-defined and analytic on the domain $B_{\epsilon}(0) \cup \{\theta \in \mathbb C : \text{Re}(\theta) > 0\}$, which implies that $\int^{\infty}_0 e^{-\theta t}dF(t)$ has an analytic continuation with a domain containing $\sigma_c=0$.  
\end{proof}

\section{Construction of a desired tier sequence under weak reversibility}\label{appdx:prove lemma}

Recall that in the proof of Theorem \ref{thm:maybe main} it was important that the states $x(i)$'s in the path $\gamma_\rho$ are in the irreducible state space $\mathbb S$. Based on this observation, we assumed $\{x_n\}\subseteq \mathbb S$ in Theorem \ref{thm: stong tier1 and irr is enough}.  If this condition is removed, non-exponential ergodicity does not necessarily holds. We demonstrate this with an example. 
\begin{example}\label{ex: need for new seq}
 Let $(\Sp, \C, \Re, \kappa)$ be the following reaction system.
    \begin{align*}
    \emptyset \xrightleftharpoons[\kappa_2]{\kappa_1} A+B 
    \end{align*}
We can check that $(\emptyset \to A+B, A+B \to \emptyset)$ is a strong tier-1 cycle along $\{(n,0)\}$. Let $x$ be an arbitrary state in $\Z^2_{\ge 0}$. For the associated Markov chain $X$ with $X(0) = x$, $\mathbb S_x = \{(a,b) \in \Z^2_{\ge 0} : a-b = (x)_1 - (x)_2\}$, which does not contain $\{(n,0)\}$. Hence Theorem \ref{thm: stong tier1 and irr is enough} cannot be applied to this example. 
Indeed, $X$ is exponentially ergodic as the state space $\mathbb S_x$ is finite for each $x$. There are many ways of verifying this. One can use that the congestion rate \eqref{eq: congestion ratio} is finite for any choice of paths $\Gamma$ on a finite state space.
    \hfill $\triangle$  
\end{example}

As we have shown in Examples \ref{ex: need for new seq} and \ref{ex: structural4}, once a strong tier-1 cycle is identified with a sequence $\{x_n\}$, it is possible that $\{x_n\}$ is not lying in the state space. In this case, if the state space has nice `reachability', we can find another sequence $\{\bar x_n\}$ by modifying $\{x_n\}$ slightly so that $R$ is still a strong tier-1 cycle. We demonstrate this idea with the same reaction network in \eqref{eq: structural4}. 

\begin{example}
    Let $(\Sp, \C, \Re)$ be the following reaction network. 
    \begin{align*}
\begin{split}
    &\emptyset \ \ \longrightarrow \ \ A+2B+C \\[-0.5ex]
    &  \quad \displaystyle \nwarrow \hspace{1.5cm}  \swarrow  \qquad \qquad \quad 2C \rightleftharpoons 3C.\\[-0.5ex]
& \hspace{0.7cm} 2A+B+2C \hspace{1.2cm} 
\end{split}
\end{align*} 
As we have shown in Example \ref{ex: structural4}, $R = (\emptyset \to A+2B+C, A+2B+C \to 2A+B+2C, 2A+B+2C \to \emptyset)$ is a strong tier-1 cycle along $\{x_n\}$, where $x_n = (n,0,0)$.  
Let $x$ be an arbitrary state such that $x>\vec{0}$.  It is uncertain that ${x_n}$ is in $\mathbb  S_x$. (indeed $\{x_n\} \not \subset \mathbb S_x$ for any $x$). But we can modify $\{x_n\}$ to get a new tier sequence so that it is contained in $\mathbb S_x$ and $R$ is also a strong tier-1 cycle along it.

We demonstrate the main idea in this example.
For simplicity let $x=\vec{0}$ be the initial state. 
First rewriting $x_n= \frac{n}{3}  (3,0,0)$ we define $\bar x_n=\lfloor \frac{n}{3} \rfloor (3,0,0)$ for each $n$, where $\lfloor a  \rfloor$ denotes the largest integer that is smaller than $ a$ for any $a\in \mathbb R$. Due to this slight modification, $x_n$ and $\bar x_n$ are close. Then one can check that $R$ is a strong tier-1 cycle along $\{\bar x_n\}$.

Now to show $\{\bar x_n\}\subseteq \mathbb S_x$, it suffices to show that $(3,0,0)$ is reachable from $x$. If so, using the same reactions for the transition from $x$ to $(3,0,0)$, we can show that $k(3,0,0)$ is reachable from $x$ for any integer $k$. In fact using the reactions in the ordered list $(\emptyset \to A + 2B + C, A+2B+C \to 2A+B+2C, A+2B+C \to 2A+B+2C, 3C \to 2C, 3C \to 2C, 3C \to 2C)$, the associated Markov chain $X$ can reach $(3,0,0)$ from $x$.

To finally obtain a tier-sequence, we select the subseqence $(3n,0,0)$ from $\{\bar x_n\}$.
\hfill $\triangle$ 
\end{example}

Now we state the lemma for constructing a desired tier-sequence from a given tier-sequence $\{x_n\}$.
\begin{lem}\label{lem:new desired tier sequence}
    Let $(\Sp,\C,\Re)$ be a weakly reversible reaction network with $\Sp = \{S_1,S_2,\dots,S_d\}$. Suppose that there exists a strong tier-1 cycle $R=(y^*_1\rightarrow y^*_2, y^*_2\rightarrow y^*_3, \cdots, y^*_M \rightarrow y^*_{M+1} )$ along a tier sequence $\{x_n\}$ that satisfies
    \begin{enumerate}
        \item  $\{x_n-x_0\} \subset \text{span}\{y'-y : y\to y'\in \Re\}$, and
        \item $y^*_{m_1} < y^*_{m_2}$ for some $m_1,m_2 \in \{1,2,\dots, M\}$.
    \end{enumerate}
        Let $x$ be an arbitrary state such that $x \ge y^*_{1}$ and $(x)_i = c_i$ for $i \in I^c_{\{x_n\}}$, where $c_i$'s are such that $(x_n)_i \equiv c_i$. Then for the state space $\mathbb S_x$ of the associated Markov chain, there exists a new tier sequence $\{\bar x_n\} \subseteq \mathbb S_x$ along which $R$ is a strong tier-1 cycle. Furthermore, $(\bar x_n)_i \equiv c_i$ for $i \in I^c_{\{x_n\}}$, and $\sup_k\lVert\bar x_k - x_{n_k}\rVert_{\infty}<\infty$ for some subsequence $\{x_{n_k}\}$ of $\{x_n\}$. 
\end{lem}
\begin{rem}
    An initial state $x$ such that $x \ge y^*_{1}$ and $(x)_i = c_i$ for $i \in I^c_{\{x_n\}}$ always exists because $x_n \ge y^*_1$ by result 1 of Proposition \ref{prop:strong tier 1}.
\end{rem}
We will set the new sequence $\bar x_n$ as
\begin{align}\label{eq:desired seq}
    \bar x_n=x+\sum_{y\to y'}d^{y\to y'}_n (y'-y) \quad \text{for some $\{d^{y\to y'}_n\}\subset \Z_{\ge 0}$.}
\end{align}
Then the first part of the proof of Lemma \ref{lem:new desired tier sequence} is to choose $\{d^{y\to y'}_n\}$ so that  $\bar x_n$ is close to $x_n$. Then Lemma \ref{lem: perturbation} below will imply that the $R$ is also a strong tier-1 cycle along $\{\bar x_n\}$. The second part of the proof is to show  reachability from $x$ to $\bar x_n$. We use Lemma \ref{lem:accessibility} below to more efficiently organize this part of the proof. 
To handle such accessibility issues, reactions forming a cycle play a critical role
as highlighted in Remark \ref{rem:why cycle}. We more precisely verify this observation in Lemma \ref{lem:cycle makes active paths}. To enhance the readability, three useful lemmas including Lemma \ref{lem:cycle makes active paths} will be introduced at the end of this section although they are used in the proof of Lemma \ref{lem:new desired tier sequence}. 

We now provide  Lemma \ref{lem: perturbation} and Lemma \ref{lem:accessibility}.
\begin{lem}\label{lem: perturbation}
    Let $(\Sp, \C, \Re)$ be a reaction network. Suppose that there exists a strong tier-1 cycle $R$ along a tier-sequence $\{x_n\}$.  Let $\{\bar x_n\}$  be a sequence such that $\sup_n \lVert x_n-\bar x_n \rVert_{\infty} < \infty$ and $(x_n)_i \equiv (\bar x_n)_i$ for each $i \in I^c_{\{x_n\}}$. Then there exists a subsequence $\{\bar x_{n_k}\}$ of $\{\bar x_n\}$ such that $\{\bar x_{n_k}\}$ is a tier sequence of $(\Sp, \C, \Re)$ and $R$ is a strong tier-1 cycle along $\{\bar x_{n_k}\}$.  
\end{lem}

\begin{proof}
    Note that $\sup_n\Vert x_n - \bar x_n \Vert_{\infty}  <\infty$ and $\{x_n - \bar x_n\} \subseteq \Z^d$. Thus we can choose a subsequence $\{x_{n_k} - \bar x_{n_k}\}$ of $\{x_n-\bar x_n\}$ so that $(x_{n_k} - \bar x_{n_k})_i$ is a constant for each $i$, especially zero for $i \in I^c_{\{x_n\}}$. Then $\bar x_{n_k}\equiv x_{n_k}+v$ for some $v\in \mathbb Z^d$ such that $(v)_i=0$ for each $i\in I^c_{\{x_n\}}$. Since $\{x_n\}$ is a tier-sequence, so $\{\bar x_{n_k}\}$ is.

    Now, by the definition of $\{\bar x_{n_k}\}$, for any $y\in \C_s$, either 
    \begin{align*}
        \lambda_y(\bar x_{n_k})\equiv \lambda_y(x_{n_k}+v) \equiv 0 \quad \text{or} \quad \lim_{n\to \infty}\frac{\lambda_y(x_{n_k})}{\lambda_y(\bar x_{n_k})}=1.
    \end{align*}
    Hence the limits in \eqref{eq:key for strong tier 1} are the same for $\{x_{n_k}\}$ and $\{\bar x_{n_k}\}$. Consequently, $R$ is a strong tier-1 cycle along  $\{\bar x_{n_k}\}$.
\end{proof}

\begin{lem}\label{lem:accessibility}
    Let $(\Sp, \C, \Re)$ be a reaction network. Let $z\in \Z^d_{\ge 0}$. Suppose that $z+w$ is reachable from $z$ for some  $w>\vec{0}$ in $\Z^d_{\ge 0}$. Then for an arbitary $\{d^{y\to y'}\in \Z_{\ge 0} : y\to y' \in \Re\}$ satisfying $\sum_{y\to y'\in \Re} d^{y\to y'}(y'-y) \ge \vec{0}$,  the state $\bar z:=z+\sum_{y\to y'\in \Re} d^{y\to y'}(y'-y)$ is reachable from $z$.
    \end{lem}
\begin{proof}
    We first show that $z+kw$ is reachable from $z$ for any $k\in \mathbb Z_{\ge 0}$. Since $z+w$ is reachable from $z$, there exists an active path $\gamma_1=(z=z(1), z(2),\dots, z(M) = z+w )$. In other words, there exists a reaction $y_m\to y'_m$ such that $z(m+1)=z(m)+y'_m-y_m$ and $\lambda_{y_m\to y'_m}(z(m))>0$ for each $m$. Then $\gamma_2:=(z(1)+w, z(2)+w,\dots, z(M)+w)$ is active since $z(m+1)+w=(z(m)+w)-(y'_m-y_m)$ and $\lambda_{y_m\to y'_m}(z(m)+w)>0$. Since the composition $\gamma_1\circ \gamma_2$ is active, $z(M)+w=z+2w$ is reachable from $z$. Inductively, we can show that $z+kw$ is reachable from $z$  for each $k\in \mathbb Z_{\ge 0}$.

    Now we show that for some $k\in \mathbb Z_{\ge 0}$ i) $\bar z +kw$ is reachable from $z$ and ii) $\bar z$ is reachable from $\bar z +kw$. Once these are shown, the result follows. 
    First, let $z+kw$ be a state with large enough $k$ so that each reaction $y\to y'$ can fire $d^{y\to y'}$ times in any order from $z+kw$. Then $\bar z+kw$ is reachable from $z+kw$. Furthermore, as aforementioned, $z+kw$ is reachable from $z$. Therefore, $\bar z+kw$ is reachable from $z$. Next, to show that  $\bar z$ is reachable from $\bar z+kw$, we first construct an active path from $\bar z$ to $\bar z+kw$ with the reactions $y_m\to y'_m$ used for $\gamma_1$. For each reaction $y_m \to y'_m$, we have 
    \begin{align*}
        \lambda_{y_m\to y'_m}\left (z(m)+\sum_{y\to y'}d^{y\to y'}(y'-y) \right )>0 
    \end{align*}
    since $\lambda_{y_m\to y'_m}(z(m))>0$ and $\sum_{y\to y'}d^{y\to y'}(y'-y)\ge \vec{0}$.
This implies that there exists an active path from $\bar z$ to $\bar z+kw$ through $z(m)+\sum_{y\to y'}d^{y\to y'}(y'-y)$'s. Finally by Lemma \ref{lem:weak rev}, weak reversibility of $(\Sp,\C,\Re)$ implies that there exists an active path from $\bar z+kw$ to $\bar z$.
\end{proof}

Now we turn to the proof of Lemma \ref{lem:new desired tier sequence}.

\begin{proof}[\textbf{Proof of Lemma \ref{lem:new desired tier sequence}}]
Let $\{\bar x_n\}$ be a sequence of the form \eqref{eq:desired seq}. Then for any $\{d^{y\to y'} : y\to y' \in \Re\}$ such that $\sum_{y\to y'\in \Re}d^{y\to y'}(y'-y)>\vec{0}$, $\bar x_n$ is reachable from $x$ by the following facts.
\begin{enumerate}
    \item[i)] $x+y^*_{m_1}-y^*_{m_0}$ is reachable from $x$ due to Lemma \ref{lem:cycle makes active paths}.
    \item[ii)] $x+y^*_{m_2}-y^*_{m_0}$ is reachable from $x+y^*_{m_1}-y^*_{m_0}$ due to Lemma \ref{lem:cycle makes active paths}.
    \item[iii)]  $\bar x_n+y^*_{m_1}-y^*_{m_0}$ is reachable form $x+y^*_{m_1}-y^*_{m_0}$ due to Lemma \ref{lem:accessibility} with $z=x+y^*_{m_1}-y^*_{m_0}$ and $w=x+y^*_{m_2}-y^*_{m_0}$.   
    \item[iv)] $\bar x_n$ is reachable from $\bar x_n+y^*_{m_1}-y^*_{m_0}$ by Lemma \ref{lem:weak rev} since $\bar x_n+y^*_{m_1}-y^*_{m_0}$ is reachable from $\bar x_n$.
\end{enumerate}
Therefore $\{\bar x_n\} \subseteq \mathbb S_x$.

Next, we will show how to choose $\{d^{y\to y'}: y\to y'\in \Re\}$ so that $\sup_k\Vert x_n-\bar x_n \Vert_\infty < \infty$ and $(x_n)_i\equiv (\bar x_n)_i$ for each $i\in I^c_{\{x_n\}}$.
The first step is to find a basis $\{w_\ell\}$ of a vector space 
\begin{align*}
    V:=\text{span}\{y'-y : y\to y'\in \Re\} \cap \{v\in \mathbb R ^d : (v)_{i} = 0, i \in I^c_{\{x_n\}} \},
\end{align*}
such that $w_\ell = \sum_{y \to y' \in \Re} (z_\ell)_{y \to y'} (y'-y)$ for each $\ell \in \{1,2,\dots,L\}$ with some integers $(z_\ell)_{y \to y'}$'s.
     Note that $\nu \in V$ if and only if 
     \begin{align}\label{eq:u in V}
         \nu = \sum_{y \to y' \in \Re} (u)_{y\to y'}(y'-y) = \sum_{i \in I_{\{x_n\}}}(v)_i e_i
     \end{align}
     for some $u \in \mathbb R^{|\Re|}$ and $v\in \mathbb R^{d}$. 
    By Lemma \ref{lem: null space}, $u$ is a linear combination of integer valued vectors $\{z_k\} \subseteq \mathbb Z^{|\Re|}$ because the vector $((u)_1,\dots,(u)_{|\Re|},(v)_1,\dots,(v)_d)^\top$ is a solution to the linear equation in \eqref{eq:u in V} (i.e. $((u)_1,\dots,(u)_{|\Re|},(v)_1,\dots,(v)_d)^\top$ is in the kernel of some integer-valued matrix). Consequently, each $\nu\in V$ can be written as
    \begin{align*}
        \nu=\sum_{y \to y' \in \Re} (u)_{y\to y'}(y'-y) &=\sum_{y \to y' \in \Re} \left (\sum_{k=1}^K a_k z_k \right )_{y\to y'} (y'-y) \\
        &= \sum_{k=1}^K a_k  \underbrace{\left ( \sum_{y\to y'\in \Re} (z_k)_{y\to y'} (y'-y)\right ) }_{w_k}
    \end{align*}
    Therefore we can find a basis $\{w_\ell\}_{\ell=1}^L$ by selecting 
     linearly independent vectors among $w_k$'s.

      
     Now for each $n$, since $x_n-x_0\in V$, we can write $x_n - x_0 = \sum_{\ell=1}^L b_n^\ell w_\ell$, for some $\{b_n^\ell\} \subseteq  \R$. Then we let  $\bar x_n= x + \sum _{\ell=1} ^L \lfloor b_n^\ell  \rfloor w_\ell$. 
     Since $w_\ell$'s are in $\{v\in \mathbb R^d: (v)_i=0, i\in I^c_{\{x_n\}}\}$ and $(x)_i=c_i$ for each $i\in I^c_{\{x_n\}}$, we have $(x_n)_i \equiv (\bar x_n)_i$ for each $i\in I^c_{\{x_n\}}$. Also for all $n$, $\Vert x_n-\bar x_n \Vert_\infty \le \Vert x_0-x \Vert_\infty+\sum_{\ell=1}^L \Vert w_\ell \Vert_\infty$. Lastly, note that $d^{y\to y'}$'s in \eqref{eq:desired seq} are supposed to be non-negative integers. Hence we first write   $\bar x_n=x+\sum_{y\to y'} c^{y\to y'}_n (y'-y)$ for some $\{c^{y \to y'}_\ell\} \subseteq \mathbb Z$ where $\{c^{y \to y'}_\ell\} \subseteq \mathbb Z$ comes from that
     \begin{align*}
         \bar x_n &= x + \sum _{\ell=1} ^L \lfloor b_n^\ell  \rfloor w_\ell = x + \sum _{\ell=1} ^L \lfloor b_n^\ell  \rfloor \left ( \sum_{y \to y' \in \Re} (z_\ell)_{y \to y'} (y'-y) \right) \notag \\
         &= x+ \sum_{y \to y' \in \Re}\left(\sum _{\ell=1} ^L \lfloor b_n^\ell  \rfloor (z_\ell)_{y \to y'}\right)(y'-y). 
     \end{align*}     
      Then we can show that $c^{y\to y'}_n$'s can be replaced with non-negative integers. Let $n$ and $y\to y'$ be a fixed index and a fixed reaction, respectively. Suppose $c^{y\to y'}_n$ is a negative integer. By Definition \ref{def:WR},  weak reversibility implies that there exists a subset $\Re_{y\to y'}$ such that $y'-y=-\left ( (y_3- y_2)+\dots +(y_M-y_{M-1})+(y_1-y_M) \right )$ for some complexes $y_i$'s. Then we can rewrite $\bar x_n$ as $\bar x_n=x+\sum_{y\to y'} d^{y\to y'}_n (y'-y)$ for some $\{d^{y\to y'}_n\} \subseteq \Z_{\ge 0}$.
\end{proof}

\begin{lem}\label{lem:cycle makes active paths}
    Let $(\Sp,\C,\Re)$ be a reaction network. For an ordered list of reactions $R=(y_1\to y_2,y_2\to y_3,\dots,y_{M-1}\to y_{M}, y_M\to y_1) \subseteq \Re$, if $x\ge y_1$, then $x+y_m-y_1$ is reachable from $x+y_k-y_1$ for any distinct $m$ and $k$. 
\end{lem}
\begin{proof}
   Note that $\lambda_{y_k\to y_{k+1}}(x+y_k-y_1)>0$ since $x+y_k-y_1>y_k$. Hence $x+y_{k+1}-y_1$ is reachable from $x+y_k-y_1$. By repeating this process, we can show that $x+y_m-y_1$ is reachable from $x+y_k-y_1$.  
\end{proof}
 We can notice that Lemma \ref{lem:cycle makes active paths} is valid in cases where either $m>k$ or $m<k$, as a result of the cyclic form of the reactions in $R$. Now, by using Lemma \ref{lem:cycle makes active paths}, we can show a property of weak reversibility in finding an active path.
 \begin{lem}\label{lem:weak rev}
     Let $(\Sp,\C,\Re)$ be a weakly reversible reaction network. For two state $z$ and $w$, if $w$ is reachable from $z$, then $z$ is also reachable from $w$. 
 \end{lem}
\begin{proof}
    As $w$ is reachable from $z$, there exists an active path $\gamma=(z,\dots,w'',w',w)$. Let $y\to y' \in \Re$ be a reaction such that $w=w'+(y'-y)$ and $\lambda_{y\to y'}(w')>0$ (i.e. $w'>y$). By Definition \ref{def:WR}, there exist complexes $ y_i$ such that $( y_1\to  y_2, y_2\to  y_3,\dots,  y_{M-1}\to  y_{M},  y_{M}\to  y_1) \subseteq \Re$, where $ y_1=y$ and $ y_2=y'$. Then by Lemma \ref{lem:cycle makes active paths}, $w'$ is reachable from $w$ regarding $m=1, k=2, w'=x=x+y_m-y_1$ and $w=x+y_k-y_1$. By using the same argument, we can show that $w''$ is reachable from $w'$ and ultimately $w''$ is reachable from $w$. Repeating this, we can show that $z$ is reachable from $w$.
\end{proof}

\begin{lem}\label{lem: null space}
    Let $A$ be a $n \times m$ an integer-valued matrix. Then the null space of $A$ has integer valued basis. 
\end{lem}
\begin{proof}
    We only use addition, subtraction, multiplication and division to get the null space of $A$, which are closed operations in $\mathbb Q$.  
\end{proof}

\end{document}